\documentclass{cmslatex}
\usepackage[paperwidth=7in, paperheight=10in, margin=.875in]{geometry}
\usepackage[backref,colorlinks,linkcolor=red,anchorcolor=green,citecolor=blue]{hyperref}
\usepackage{amsfonts,amssymb}
\usepackage{amsmath}
\usepackage{graphicx,mathrsfs}
\usepackage{graphicx}
\usepackage{cite}
\usepackage{enumerate}
\usepackage{psfrag}
\usepackage{mathtools}
\mathtoolsset{showonlyrefs}

\newcommand{\bfx}{{\bf x}}
\newcommand{\bfX}{{\bf X}}

\newcommand{\ud}{\mathrm{d}}

\renewcommand{\theequation}{\arabic{section}.\arabic{equation}}

   \allowdisplaybreaks

\begin{document}
\title{Asymptotic gradient flow structures of a nonlinear Fokker--Planck equation\thanks{Received date \today}}

\author{Maria Bruna\thanks{Mathematical Institute, University of Oxford, Oxford OX2 6GG, UK (bruna@maths.ox.ac.uk).}
\and 
Martin Burger\thanks{Institut f\"ur Numerische und Angewandte Mathematik and Cells in Motion Cluster of Excellence, Westf\"alische Wilhelms Universit\"at M\"unster, Einsteinstrasse 62, 48149 M\"unster, D (martin.burger@wwu.de).}
\and
Helene Ranetbauer\thanks{University of Vienna, Faculty of Mathematics, Oskar-Morgenstern-Platz 1, 1090 Wien, AT (helene.ranetbauer@univie.ac.at).}
\and
Marie-Therese Wolfram\thanks{University of Warwick, Coventry CV4 7AL, UK and RICAM, Austrian Academy of Sciences, Altenbergerstr. 66, 4040 Linz, AT (m.wolfram@warwick.ac.uk).}
}

\pagestyle{myheadings} \markboth{ASYMPTOTIC GRADIENT FLOWS OF A FOKKER--PLANCK EQUATION}{M. BRUNA, M. BURGER, H. RANETBAUER AND M.T. WOLFRAM} \maketitle

\begin{abstract}
   In this paper we consider a nonlinear Fokker--Planck equation with asymptotically small parameters. It describes the diffusion of finite-size particles in the presence of a fixed distribution of obstacles in the limit of low-volume fraction. The equation does not have a gradient flow structure, but can be interpreted as an asymptotic gradient flow, that is, as a gradient flow up to a certain asymptotic order. We use this scalar equation as a simple testbed model for more complicated systems of this kind. We discuss several possible entropy-mobility pairs, illustrate their dynamics with numerical simulations, present global in time existence results and study the long time behavior of solutions. 
\end{abstract}


\section{Introduction.} \label{sec:introduction}

In this paper we study the solution of a nonlinear Fokker--Planck equation of the form
\begin{align}\label{pme_d}
\partial_t r &= \nabla  \cdot \left[   ( 1 +
 \varepsilon_1   r -\varepsilon_2 b) \nabla {  r} + \varepsilon_3 r \nabla { b}  \right],   \qquad t>0, \bfx \in \Omega,
\end{align}
where $r = r(\bfx, t)$ describes the density of interacting particles diffusing through a porous medium represented by a fixed porosity density 
$b(\bfx)$. Here $\varepsilon_i >0$ are (small) parameters related to the excluded-volume interactions and $\Omega \subset \mathbb R^d$ with $d=2,3$. 
Equation \eqref{pme_d} was derived as the macroscopic limit of a microscopic system of two types of interacting particles, namely diffusing (Brownian) red and immobile (obstacle) blue particles using an asymptotic method in the limit of low volume fraction \cite{Bruna:2015eh}. While the microscopic system has a natural gradient flow (GF) structure, this is lost in the macroscopic equation \eqref{pme_d}. In this paper we are interested in a framework for generalizing the concept of GFs to extend to equations such as \eqref{pme_d}. 
 
A GF structure provides a natural framework to study the behavior of solutions \cite{AGS2008}, and as such it is a desirable feature for the macroscopic-level partial differential equation (PDE). In general, the GF structure of a conservation law is given by 
\begin{equation}\label{e:gf}
\partial_t r=\nabla \cdot \left[m(r)\nabla \frac{\delta E}{\delta r}(r)\right],
\end{equation}
where $E$ is an entropy and $m$ a mobility. Many well known PDEs can be written in this form, including the porous medium \cite{vazquez2007} or the 
linear Fokker--Planck equation \cite{JKO1998}. The decay of the entropy functional along its solutions provides useful estimates to study global in time existence and the long time behavior of 
solutions. GF techniques as well as entropy methods have been used successfully to understand the structure 
of many nonlinear PDEs and their qualitative behavior \cite{jungel2015boundedness, MV99,zamponi2015analysis,ZM2015}.  
The interpretation of \eqref{e:gf} as a GF with respect to the Wasserstein metric was first analyzed by Jordan, Kinderlehrer and Otto in the case of the diffusion equation  
\cite{JKO1998}. This connection initiated a lot of research in the optimal transportation community, see for example \cite{liero2013gradient,L2009}. 

More recently the connection between the 
time-discrete Wasserstein flow formulation of the diffusion equation and large deviation principles was established by Adams and co-workers in \cite{adams2013,adams2011}. 
This limiting passage from the microscopic to the macroscopic level is only well understood under certain scaling assumptions on the number and size of particles. These assumptions do not hold in the case of finite-size particles as is the case of \eqref{pme_d}. The only rigorous results available are restricted to spatial dimension one, see for example \cite{Rost1984}. Different formal approaches, for example the derivation from a simple exclusion process on a discrete lattice using Taylor expansion, were presented and analyzed in \cite{burger:2015uk,burger2012nonlinear,painter2009}. The derived macroscopic equations differ to the ones for off-lattice diffusion obtained using the methods of matched asymptotics \cite{bruna2012diffusion,Bruna:2012cg}, and in some cases they may also not exhibit a full GF structure.

The lack of GF structure at the macroscopic level can result from the nature of the approximations made when passing from the particle level to the continuum. This is the case of some models such as \eqref{pme_d}. Specifically, \eqref{pme_d} was derived using the method of matched asymptotics with a small parameter $\varepsilon>0$ (such that $\varepsilon_i / \varepsilon$  for $i=1, 2, 3$ are all order one), with terms computed up to order $\varepsilon$. The truncation of the expansion at a given order in the small parameter $\varepsilon$ can result in the loss of the GF structure, and motivated our definition of an asymptotic gradient flow (AGF) \cite{bruna2016cross}. 
\begin{definition}[Asymptotic gradient flow structure]\label{def:agf}
Consider a conservation law given as an asymptotic expansion in $\varepsilon >0$ up to order $k$,
\begin{align}\label{agm_intro}
\partial_t r = \nabla \cdot [F(r;\varepsilon)] = \nabla \cdot \left [ F^{(0)}(r) + \varepsilon  F^{(1)}(r) + \cdots + \varepsilon^k F^{(k)}(r) \right].
\end{align}
Then we call \eqref{agm_intro} an asymptotic gradient flow (AGF) at order $\varepsilon^k$ if we can find a mobility $m$ and entropy $E$ such that 
\begin{align*}
m(r;\varepsilon) \nabla \frac{\delta E}{\delta r}(r;\varepsilon) = F(r;\varepsilon)  + O(\varepsilon^{k+1}). 
\end{align*}	
\end{definition}
For example, in our case of interest, \eqref{pme_d} is an equation up to order $\varepsilon$ ($k=1$). Note that GF structures are always preserved at order $k=0$, in our case this is rather trivial since we would arrive at the linear diffusion equation.

In this work we continue our efforts to develop an analytic framework for this kind of equations 
by considering a special case of interacting species: hard-sphere particles diffusing through immobile obstacle particles. The resulting equation \eqref{pme_d} is a nonlinear scalar Fokker--Planck equation that lacks a full GF structure. Its existence and uniqueness can be studied using classical existing 
results for scalar nonlinear PDEs, see for example \cite{ladyzhenska}. However, here we are interested in using \eqref{pme_d} as a testbed for employing different AGF structures and generalizing GF techniques. In particular, we want to explore how one can exploit the AGF structure to understand properties of its solutions such as existence, uniqueness, or long-time behavior. Moreover, we want to understand how to choose a given AGF structure. 
The dynamic and equilibration behavior of a full GF structure is fully determined by its entropy and mobility. In contrast,  for an AGF this is only true up to a certain order. This additional freedom allows us to use different entropy-mobility pairs for analyzing the behavior of solutions. We discuss how the various choices give different global in time existence results as well as equilibrium solutions. A natural question that arises in this context is the what is the "best" choice of an entropy if there are multiple ones, an issue we try to tackle in our case by comparison of dynamics with simulations of the original stochastic particle systems used to derive the asymptotic model.

This paper is organized as follows: we introduce the mathematical model in Section \ref{sec:model} and discuss alternative GF structures that are AGFs of \eqref{pme_d}. In Section \ref{sec:numerics} we present several numerical examples to illustrate the behavior of the solutions to \eqref{pme_d} and the GFs presented in the previous section. We also compare the solutions of the different models with Monte Carlo simulations of the underlying system of interacting particles, to explore which AGF structure might represent the underlying system best. In Section \ref{sec:GF_analysis} we consider two of the GF structures and study their advantages and disadvantages for the analysis of their solutions. In particular, we discuss existence and long time behavior. Equation \eqref{pme_d} is then discussed in Section \ref{sec:AGF_analysis}, where we present a global existence result. Finally, we conclude in Section \ref{sec:conclusion} with a summary of our results concerning the particular model \eqref{pme_d} and a discussion of the open questions and challenges for the analysis and numerical solution of AGF structures in general. 

\section{Model and Asymptotic Gradient Flow Structures.} \label{sec:model}

Below we summarize the assumptions leading to the nonlinear Fokker--Planck equation \eqref{pme_d}, as they will influence our choices of functional spaces and set of solutions later on. 
The equation is obtained as the macroscopic limit of a stochastic system with two types of particles: $N_r$ diffusing red particles of diameter $\epsilon_r$ (with diffusion coefficient normalized to one) and $N_b$ fixed blue obstacle particles of diameter $\epsilon_b$. The particles interact via hard-core collisions, which means that their centers $\bfX_i$ cannot get closer than the sum of their radii, $\| \bfX_i - \bfX_j \| \ge (\epsilon_i + \epsilon_j)/2$, where $\epsilon_i$ denotes the diameter of the particle at position $\bfX_i$. Accordingly, the red particles diffuse in a ``perforated domain'' with obstructions of diameter $\epsilon_{rb} = (\epsilon_r + \epsilon_b)/2$, distributed according to the probability density $b(\bfx)$. This is related to the porosity $\phi_b(\bfx)$ by
\begin{equation} \label{porosity_b}
\phi_b(\bfx) = N_b v_d(\epsilon_b) b(\bfx),	
\end{equation}
where $v_d(\epsilon)$ is the volume of a $d$-dimensional ball of diameter $\epsilon$. The population of red particles is characterized by the one-particle marginal probability density $r(\bfx, t)$, which describes the probability that one of the red particles is at position $\bfx$ at time $t$. The number density (describing the probability of finding \emph{any} of the red particles at a given position) is given by  $N_r r(\bfx, t)$, and the volume concentration of red particles is given by 
\begin{equation} \label{conc_r}
	\phi_r(\bfx,t) = N_r v_d(\epsilon_r) r(\bfx,t). 
\end{equation}

The procedure adopted in \cite{Bruna:2015eh} to derive the equation \eqref{pme_d} for $r$ is a systematic asymptotic expansion for  $N_b v_d(\epsilon_b) + N_r v_d(\epsilon_r) \ll 1$. 
Specifically, the parameters in \eqref{pme_d} are given by\footnote{We note there was a typo in Eq. (3.6b) of \cite{Bruna:2015eh}: there should have been a $2^{d+1}$ in the numerator of the $p(\bfx)$ term, rather than a $2^{d-1}$ in the denominator.}
\begin{equation} \label{deltas}
	\varepsilon_1 = 4(N_r-1) (d-1)   v_d(\epsilon_r), \qquad \varepsilon_2 = 4 N_b v_d(\epsilon_{rb}), \qquad \varepsilon_3 = (d-1)\varepsilon_2.
\end{equation}
Taking the particle diameters to be of a similar order, $\epsilon_r \sim \epsilon_b \sim \epsilon$, and the number of red particles large such that $N_r -1 \sim N_r$, we can write
\begin{equation*}
	\varepsilon_1 + \varepsilon_3 \sim 4 (d-1) \left [N_r v_d(\epsilon_r) + N_b v_d(\epsilon_b) \right] \sim 4(d-1) (N_r+N_b) v_d(\epsilon),
\end{equation*}
that is, $\varepsilon_1 + \varepsilon_3$ is approximately $4(d-1)$ times the volume fraction occupied by particles.

Due to the finite size of particles, there is a physical constraint on the total volume occupied by particles, $\Phi = [N_b v_d(\epsilon_b) + N_r v_d(\epsilon_r)]/|\Omega|$,  as well as a constraint on the local volume concentration
\begin{equation}
	\phi (\bfx, t) = \phi_b(\bfx) + \phi_r(\bfx, t).
\end{equation}
In particular, we have that $\Phi < \Phi_d^*$, where $\Phi_d^*$ is the maximal packing density in $d$ dimensions, $\Phi_2^* \approx 0.91$ and $\Phi_3^* \approx 0.74$.
We note that this physical constraint is not present in \eqref{pme_d}, that is, one could in principle solve \eqref{pme_d} for a very large concentration of red particles without running into mathematical problems. However, since this equation represents the physical system described above, in our analysis we will impose that its solutions remain in a set of ``feasible'' configurations. In particular, our bound for the maximum allowed density will be much lower than $\Phi_d^*$, so that we stay within the region of validity of  the model: we will require $\varepsilon_1 r + \varepsilon_3 b \le 1$, so that the total volume fraction for $d= 2$ is under 25\% (see Section \ref{sec:GF_analysis}).

Equation \eqref{pme_d} can be seen as a particular case of the model derived in \cite{bruna2012diffusion} and analyzed in \cite{bruna2016cross}, consisting of a cross-diffusion system of interacting particles in which both species are allowed to diffuse in the domain under an external potential. When the red moving particles are point particles ($\varepsilon_1 \equiv 0$), one recovers a linear diffusion-advection equation in a porous medium with porosity $\phi_b$. When the obstacles are removed ($\varepsilon_2 = \varepsilon_3 \equiv 0$), equation \eqref{pme_d} is a nonlinear diffusion equation describing the enhanced diffusion due to self-excluded volume \cite{Bruna:2012cg}.

\subsection{Alternative asymptotic gradient flow structures.} \label{sec:AGF}
As mentioned in the introduction, \eqref{pme_d} does not have a formal gradient flow (GF) structure. However, writing $\varepsilon_i \sim \varepsilon$ for all $i$, we can see \eqref{pme_d} as an asymptotic expansion in $\varepsilon$ up to order one, and express it as an asymptotic gradient flow (AGF)
\begin{align}\label{pme_gf}
\partial_t r &= \nabla \cdot \left[m(r; \varepsilon )\nabla\frac{\delta E}{\delta r} (r;\varepsilon)+  f(r,\nabla r; \varepsilon)\right],\qquad f = O(\varepsilon^2).
\end{align}

Note that since equation \eqref{pme_d} was derived as an asymptotic expansion, formally there should be a term $ O(\varepsilon^2)$ added to the right-hand side. This means we do not know what the second-order term is: it could be zero, or it could be anything else. This is why, in the definition of AGF, the higher-order term $f$ can be arbitrary as long as it is of order $\epsilon^2$ (see Definition \ref{def:agf}). In what follows we discuss several possible entropy-mobility pairs that induce an AGF of the form \eqref{pme_gf} for \eqref{pme_d}. These obviously also induce a GF (corresponding to having $f= 0$ in \eqref{pme_gf}). With a slight abuse of language, throughout this paper will refer to equation \eqref{pme_d} as the AGF to distinguish it from the full GFs. But it should be clear that it is just a choice among infinitely many, that correspond to setting the $ O(\varepsilon^2)$ to zero in \eqref{pme_d}.

For ease of notation, we omit the dependence of $m$ and $E$ on $\varepsilon$ from now on.
The first entropy-mobility pair $(E_1,m_1)$ is
\begin{subequations}
	\label{pair1}
	\begin{align} \label{entropy}
E_1(r) &=\int_{\Omega} \left[r(\log r - 1)+\frac{1}{2}\varepsilon_1 r^2+\varepsilon_3 rb \right] \ud \bfx,\\
 \label{mobility}
m_1(r) &= r (1-\varepsilon_2 b),
\end{align}
with the  higher-order term in $\varepsilon$ given by
\begin{equation} \label{hot_f}
f_1(r,\nabla r)=rb (\varepsilon_1 \varepsilon_2\nabla r+\varepsilon_2 \varepsilon_3 \nabla b).
\end{equation}
\end{subequations}
The corresponding entropy variable is
\begin{equation} \label{ent_var1}
	u_1 = \frac{\delta E_1}{\delta r} = \log r + \varepsilon_1 r + \varepsilon_3 b,
\end{equation}
and $m_1 \nabla u_1 =  ( 1 + \varepsilon_1 r - \varepsilon_2 b) \nabla {  r} + \varepsilon_3 r \nabla { b} + O(\varepsilon^2)$ as required. 

The equilibrium solution $r_\infty$ of the full GF (corresponding to $f\equiv 0$  in \eqref{pme_gf}) can be found as the minimizer of the entropy $E$, and it agrees with the stationary solution $r_*$ of \eqref{pme_gf} up to order $\varepsilon$. To find the equilibrium solution $r_{1,\infty}$ corresponding to the first entropy $E_1$, we look for an asymptotic expansion $r_{1,\infty} \sim r^{(0)} + \varepsilon r^{(1)} + \cdots$. Setting $u_1 = \chi$ constant in \eqref{ent_var1} and using the normalization condition on $r$ we obtain, at leading order, $r^{(0)} = 1$. At order $\varepsilon$, we find $
r^{(1)} + \bar \varepsilon_1 + \bar \varepsilon_3 b = \chi_1$, where $\bar \varepsilon_i := \varepsilon_i / \varepsilon$ and $\chi_1$ is a constant. Thus, $r^{(1)} = \tilde \chi_1 - \bar \varepsilon_3 b$ or $r^{(1)} =  \bar \varepsilon_3 (1- b)$ after using again the normalization condition ($\int r^{(1)} = 0$). Therefore, we find that the equilibrium solution of the first GF pair is
\begin{equation} \label{rinf1}
	r_{1,\infty} = 1 + \varepsilon_3 (1-b) + O(\varepsilon^2).
\end{equation}

As mentioned before, the lack of GF structure at the macroscopic level is due to the fact that \eqref{pme_d} corresponds to an asymptotic expansion truncated at a given order. One could argue that this should be resolved by calculating the next order terms in the expansion, and that these should coincide with \eqref{hot_f}. However, $f_1$ only contains some of the possible order $\varepsilon^2$ terms (corrections to pairwise interactions), but in general also the first-order correction to triplet interactions should be at that order. In other words, in general calculating the next order term in \eqref{pme_d} would force us to write down expansions up to order $\varepsilon^2$ for the mobility and the entropy, potentially leading to new terms of order $\varepsilon^3$ in \eqref{hot_f}. 

In terms of an AGF structure, the question of the ``right'' entropy functional and mobility matrix arises. Indeed, just as with standard GF where one sometimes can define more than one entropy-mobility pair, is the AGF non-uniquely defined? To illustrate this consider a second entropy-mobility pair $(E_2, m_2)$
\begin{subequations}
	\label{pair2}
\begin{align}\label{entropy2}
E_2(r)&=\int_\Omega r\left[\log\left(\frac{r}{1-\varepsilon_1 r -\varepsilon_3 b}\right)-1\right] \,\ud \bfx,\\
\label{mobility2}
m_2(r)&=r(1-\varepsilon_1 r-\varepsilon_2 b).
\end{align}
The higher-order term in $\varepsilon$ is given by
\begin{equation}\label{f_GF2}
	f_2(r, \nabla r) = -\varepsilon_1 r^2 \left( \varepsilon_1 \nabla r + \varepsilon_3 \nabla b \right) +  (\varepsilon_2 - \varepsilon_3) r b \left( 2 \varepsilon_1 \nabla r + \varepsilon_3 \nabla b \right)  + O(\varepsilon^3). 
\end{equation}
\end{subequations}
In this case we have the following entropy variable
\begin{align} \label{ent_var2}
u_2 &= \frac{\delta E_2}{\delta r}=\log r -\log (1-\varepsilon_1 r-\varepsilon_3 b) +\frac{\varepsilon_1 r}{1-\varepsilon_1 r-\varepsilon_3 b},
\end{align}
and as before we find that $ m_2 \nabla u_2 =   ( 1 + \varepsilon_1   r -\varepsilon_2 b) \nabla {  r} + \varepsilon_3 r \nabla { b} + O(\varepsilon^2)$. 
This implies that the second pair \eqref{pair2} is also an AGF to equation \eqref{pme_d}, and that errors between using a GF representations with $(E_1, m_1)$, $(E_2, m_2)$, or the AGF \eqref{pme_gf} will be, at most, at $O(\varepsilon^2)$. Does the same hold when considering the equilibrium solutions of the GFs or the stationary solution of the AGF? The equilibrium solution $r_{2,\infty}$ corresponding to the second pair \eqref{pair2} is given by imposing $u_2 = \chi$ constant, which expanding \eqref{ent_var2} asymptotically leads to 
\begin{equation} \label{u2_expanded}
	u_2 \sim \log r_{2,\infty} + 2\varepsilon_1 r_{2,\infty} + \varepsilon_3 b = \chi. 
\end{equation}
Although it may seem that there is a difference at $O(\varepsilon)$ since the terms with $\varepsilon_1$ between $u_1$ and $u_2$ differ (compare \eqref{u2_expanded} with \eqref{ent_var1}), that is not the case:  note that the term with $\varepsilon_1$ in \eqref{ent_var1} does not affect the equilibrium solution \eqref{rinf1} up to $O(\varepsilon)$. For the same reason, the equilibrium solution $r_{2,\infty}$ of the second entropy agrees with that of the first entropy, \eqref{rinf1}, up to order $O(\varepsilon)$.

This motivates the following definition.
\begin{definition}
Two AGF structures defined by a entropy-mobility pair $(E_i, m_i)$  and entropy variables $u_i = \delta E_i/\delta r$ are equal up to order $\varepsilon^k$, if
\begin{enumerate}[(1)]
\item the asymptotic expansions $m_i \nabla u_i$ are equal up to order $\varepsilon^k$.
\item the asymptotic expansions of their corresponding stationary solutions, found setting $u_i = \chi_i$ constant, are equal up to order $\varepsilon^k$.\end{enumerate}  
\end{definition}

Given these observations, we can define the following  family of AGFs up to $O(\varepsilon)$ for equation \eqref{pme_d}:
\begin{subequations}
	\label{pair_gen}
	\begin{align}\label{entropy_gen}
E_3(r)&=\int_\Omega r\left[\log\left(\frac{r}{1-\alpha \varepsilon_1 r - \varepsilon_3 b}\right)-1\right] \,\ud \bfx,\\
m_3(r)&=r(1- \beta \varepsilon_1 r -  \varepsilon_2 b),
\end{align}
\end{subequations}
with $2 \alpha - \beta = 1$. So for example we could choose to have a mobility as in the first pair \eqref{mobility}, setting $\beta=0$ and giving $1-\varepsilon_1 r/2 -\varepsilon_3 b$ in the denominator of the log term in the entropy. 

A key feature of the Fokker--Planck equation \eqref{pme_d} that allowed us to find several entropy-mobility pairs is that the drift term is at order $\varepsilon$. It is for this reason that the stationary solution is constant at leading order and the coefficient $\varepsilon_1$ of the nonlinear diffusion does not enter the solution until second order (see \eqref{rinf1}). Below we discuss how the situation would change if instead the potential term $\nabla b$ appeared at leading order in \eqref{pme_d}. 
Suppose we have an equation of the form 
\begin{equation*}
\partial_t r=\nabla \cdot \left\{ \big[d^{(0)}(r) + \varepsilon d^{(1)}(r) + \cdots \right]\nabla r + \left[c^{(0)}(r) + \varepsilon c^{(1)}(r) + \cdots \big] \nabla b(\bfx) \right \}, 
\end{equation*}
with $c^{(0)}(r) \ne 0$. We look for an AGF structure with expansions for mobility and entropy as follows
\begin{align*}
	m(r) = m^{(0)}(r) + \varepsilon m^{(1)}(r) + \cdots,\qquad 
	e(r) = e^{(0)}(r) + \varepsilon e^{(1)}(r) + \cdots + \int r b ~\ud \bfx.
\end{align*}
Note that the last term in $e(r)$ is the natural potential energy term given a drift $\nabla b$.
We immediately obtain through the potential term that
\begin{equation*}
m^{(i)} = c^{(i)}, \quad i \ge 0.
\end{equation*}
Then, having $m^{(0)}$ fixes in turn $e^{(0)}$ through the relation
\begin{equation*}
m^{(0)} \nabla \frac{\delta e^{(0)}}{\delta r}= d^{(0)} \nabla r.
\end{equation*}
Comparing the first order coefficients, we obtain
\begin{equation*}
m^{(1)} \nabla \frac{\delta e^{(0)}}{\delta r} + m^{(0)} \nabla \frac{\delta e^{(1)}}{\delta r} = c^{(1)} \nabla r,
\end{equation*}
and similarly for the higher-order terms. Thus, we see that in the case of an order one potential term, there is less flexibility in choosing the mobility and entropy of an AGF.

We have seen that there are several possible entropy-mobility pairs resulting in different AGFs for equation \eqref{pme_d}.
The first pair \eqref{pair1} had the natural entropy, with the quadratic terms coming from the pairwise interactions accounted for in \eqref{pme_d} and the natural potential energy term. But a somewhat unnatural mobility, since it is independent of the crowding of red particles. 
The second pair \eqref{pair2} has instead a more natural mobility, which is reduced by crowding from both species, and an entropy with the advantage that gives the correct physical bounds on the density. In the next section we use numerical simulations of \eqref{pme_d} and the GFs induced by the first two pairs to highlight their differences. Then in Section \ref{sec:GF_analysis} we will demonstrate the advantages of each pair from the analysis perspective.

\section{Numerical simulations.} \label{sec:numerics}

We start with a numerical investigation of the behavior of solutions of the AGF equation \eqref{pme_d} and associated GF equations of the form \eqref{e:gf} with the entropy-mobility pairs discussed in the previous section. To assess the accuracy of the models, we compare the numerical solutions of the PDEs with stochastic simulations of the underlying microscopic model. In particular, we are interested in comparing the decay in relative entropy in each of the models. 

\begin{definition}
The relative entropy functional is defined by
\begin{align} \label{bregman0}
E^*(r(t)) = E(r(t)) - E(r_*) - \int_\Omega u(r_*) (r(t)-r_*) \, \ud \bfx,
\end{align}
where $u = \delta E/\delta r$ is the entropy variable and $r_*$ is the stationary solution.
\end{definition}

Note that, if $r$ is the solution of a GF, the stationary solution is an equilibrium solution, $r_* = r_\infty$. Then the last term in \eqref{bregman0} vanishes since $u(r_\infty) = \chi$ constant, noting that $r$ has constant mass. However, in an AGF such as \eqref{pme_gf}, $r_*$ is not a minimizer of the entropy functional $E$ and therefore the last term in \eqref{bregman0} plays a key role, which is similar to the analysis of linear Fokker-Planck equations with non-constant drift (cf. \cite{arnold2008}). For a discussion of further properties of such relative entropies, also called Bregman distances in convex optimization we refer to \cite{Burger:2016hu}.

The time-dependent solution of the full GF \eqref{e:gf} is obtained numerically using the finite-volume scheme in space described in \cite{Carrillo:2015cf}. This scheme is second-order in space and preserves non-negativity and entropy dissipation. We use a Runge--Kutta scheme in time.  To obtain the time-dependent solution of the original equation \eqref{pme_d}, we consider it in its AGF form \eqref{pme_gf} and adapt the finite-volume scheme to include the higher-order term $f$.   

The equilibrium solution $r_\infty$ of the GF equation \eqref{e:gf} is easily obtained by solving $u(r_\infty) = \chi$ constant. For example, for the first entropy $E_1$ we have (see \eqref{ent_var1})
\begin{align} \label{equil_num}
u_1 (r_\infty) = \log r_\infty + \varepsilon_1 r_\infty  + \varepsilon_3 b = \chi, \qquad
\int_\Omega r_\infty\, \ud \bfx  = 1,
\end{align}
where $\chi$ is a constant to be determined by imposing the normalization constraint.  Equations \eqref{equil_num} are solved by discretizing in space using equally spaced points and solving the resulting system of equations by the Newton--Raphson's method. In the AGF case, the stationary solution $r_* (\bfx)$ of \eqref{pme_d} is obtained as the long-time limit solution, $r_* = \lim_{t \to \infty} r(t)$ by running the time-dependent simulation for large times until it has equilibrated. Note that \emph{a priori} we do not know if this limit exists, since \eqref{pme_d} is not a GF. However, we know that \eqref{pme_d} has a unique stationary solution $r_*$ and that this is close to $r_\infty$, specifically at a distance $O(\varepsilon^2)$. This is a consequence of a more general result of \cite{bruna2016cross}, where it was proven that for the corresponding cross-diffusion system (allowing the obstacles' density $b$ in \eqref{pme_d} to be dynamic), there exists a unique stationary solution $r_*$ with $\|r_*-r_\infty\|=O(\varepsilon^2)$.

The underlying microscopic model of \eqref{pme_d} is (see \cite{Bruna:2015eh} for details) 
\begin{equation} 
\label{ssde}
\ud {\bf X}_i(t) =  \sqrt{2} \, \ud{\bf W}_i(t),
\end{equation}
where $\bfX_i(t)$ is the position of the $i$th red particle at time $t$ ($ 1 \le i \le N_r$) and ${\bf W}_i(t)$ denotes a $d$-dimensional Brownian motion. The position $\bfX_i(t)$ is constrained by hard-core interactions with other red particles, $\| \bfX_i(t) - \bfX_j(t) \| \ge \epsilon_r$ for $j\ne i$, as well as with the blue obstacles, $\| \bfX_i(t) - {\bf O}_j\| \ge \epsilon_{rb}$ for $1\le j \le N_b$. Here ${\bf O}_j$ is the (fixed) position of the $j$th obstacle. In addition, the domain boundaries $\partial \Omega$ are hard-walls. Equation \eqref{ssde} is integrated using the Euler--Maruyama method and a constant time-step $\Delta t$, and simulated using the open-source C\texttt{++} library Aboria \cite{aboria,Robinson:2017vxa}. For all simulations we used a time-step of $\Delta t = (0.5 \epsilon_r)^2/2$, leading to an average diffusion step size of $0.5 \epsilon_r$, and $10^5$ realizations of \eqref{ssde} to compute the histograms of $r(\bfx,t)$. The stationary distribution of the particle system is obtained using the Metropolis--Hastings method (the single-particle local move variant, which means one particle is picked at random and a candidate new position is drawn from a normal distribution centered around  its current position). The variance of the move is adjusted so that the acceptance rate is of the order of 23\%  \cite{roberts1997weak}. For all Metropolis--Hastings simulations we used $2\times 10^4$ realizations of the porosity distribution and $10^5$ moves per realization. To speed-up convergence of the Markov chain, we initialized the red particles according to the stationary solution up to first order given in \eqref{rinf1}. 

We perform simulations for the two-dimensional case ($d=2$). For ease of comparison, we consider one-dimensional initial data and density of obstacles $b$ in $x$ such that $r(\bfx, t)$ is also one-dimensional in $x$. In particular, all PDE simulations are performed in the domain $[-0.5, 0.5]$, using 1000 spatial grid-points and a time-step of $10^{-6}$. However, the stochastic simulations are still performed in the full two-dimensional domain $\Omega = [-0.5, 0.5]^2$. 

\def \scc {0.6}
\def \scl {.8}
\begin{figure}
\unitlength=1cm
\begin{center}
\vspace{3mm}
\psfrag{x}[][][\scl]{$x$} \psfrag{t}[][][\scl]{$t$} \psfrag{r}[][][\scl]{$ r(x,t)$} \psfrag{E}[][][\scl]{$E$} \psfrag{RelE}[][][\scl]{$E^*$} \psfrag{t}[b][][\scl]{$t$}
\psfrag{a}[][][\scl]{(a)} \psfrag{b}[][][\scl]{(b)} \psfrag{c}[][][\scl]{\quad (c)} \psfrag{d}[][][\scl]{(d)}
\psfrag{data1}[][][\scl]{\ Sim.}
\psfrag{equil1}[][][\scl]{$r_{1,\infty}$}
\psfrag{equil2}[][][\scl]{$r_{2,\infty}$}
\psfrag{ss}[][][\scl]{\ \ $r_*$}

\psfrag{AGF}[][][\scl]{AGF}
\psfrag{GF1}[][][\scl]{GF1}
\psfrag{GF2}[][][\scl]{GF2}

\psfrag{Eagradf}[][][\scl]{$E$ AGF}
\psfrag{Egammagf}[][][\scl]{$E - \gamma$ AGF}
\psfrag{Egradf}[][][\scl]{$E$ GF1}
\psfrag{Esimul}[][][\scl]{$E$ Sim.}
\psfrag{Egamsim}[][][\scl]{$E-\gamma$ Sim.}
	\includegraphics[height = .4\textwidth]{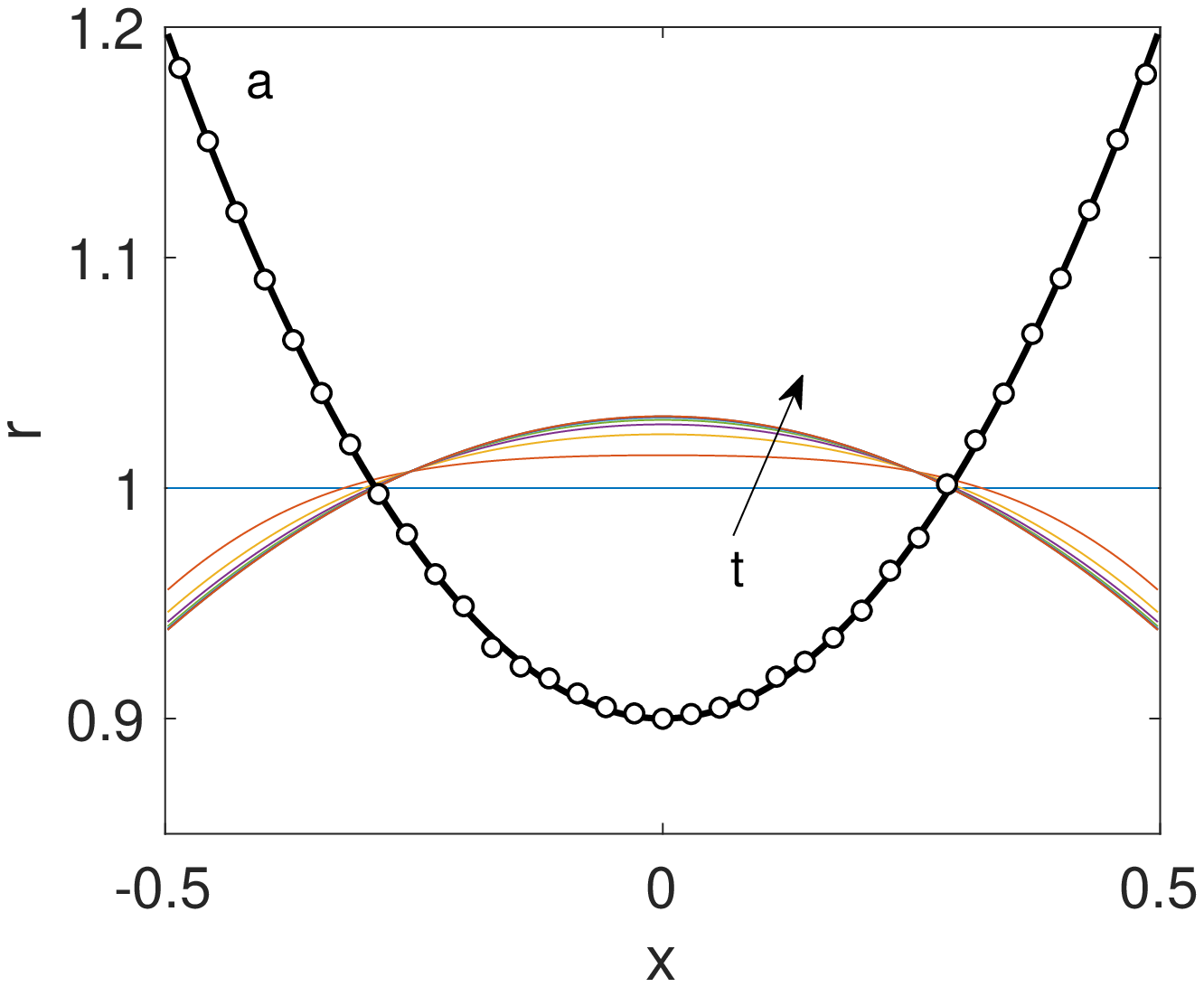} \hfill 
	\includegraphics[height = .4\textwidth]{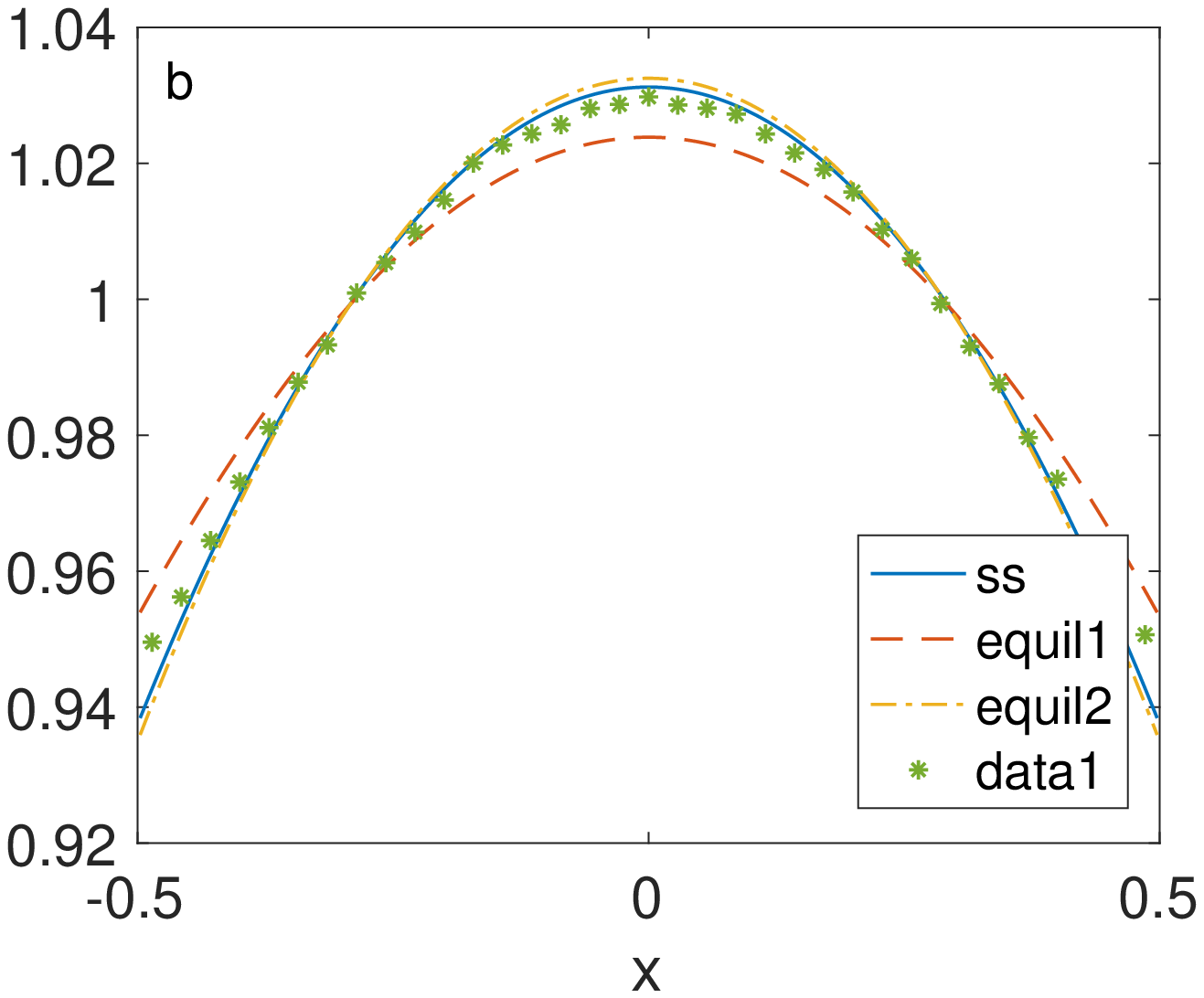}\\ \vspace{3mm}
	\includegraphics[height = .4\textwidth]{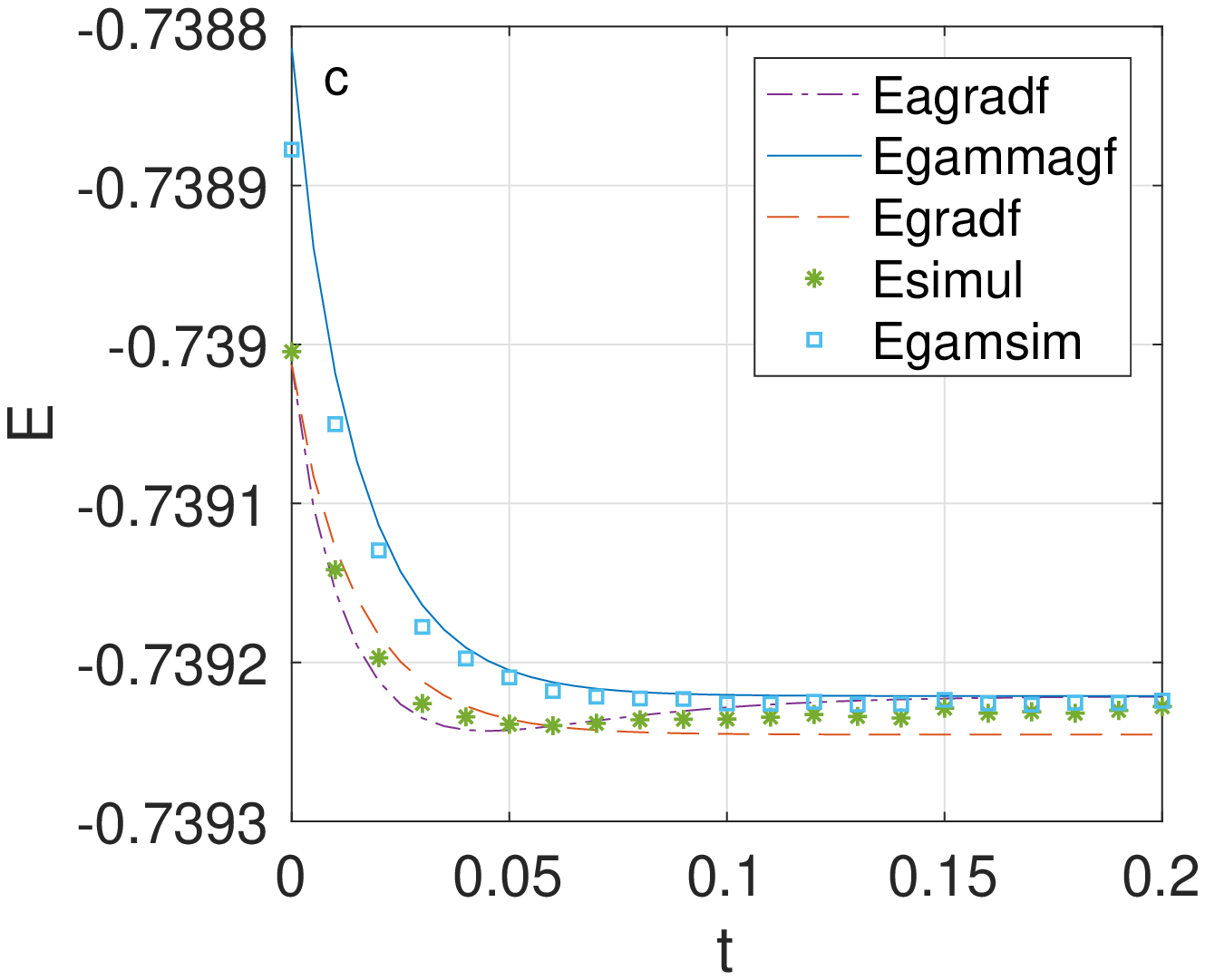} \hfill
	\includegraphics[height = .4\textwidth]{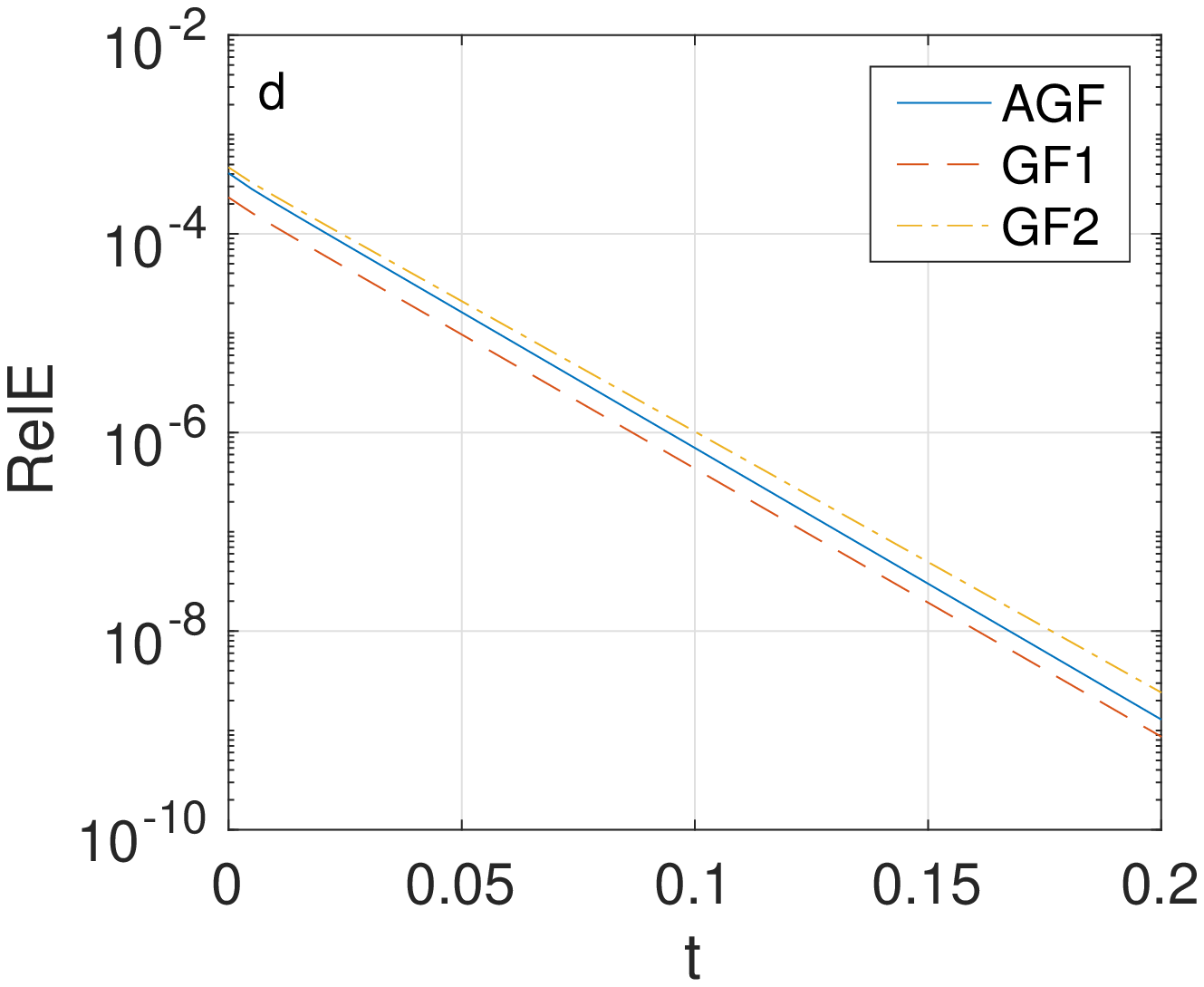} 
\caption{Example with a convex porosity distribution. (a) Time evolution of \eqref{pme_d} with a uniform initial density (times shown: $t = 0, 0.025, 0.05, \dots, 0.2$), and obstacles distribution $b(x) = 0.3(4x^2 + 3)$ (circles obtained from stochastic simulations). (b) Stationary solution $r_*$ of \eqref{pme_d}, equilibrium solution of \eqref{e:gf} with $E_1$ \eqref{entropy}, $r_{1,\infty}$ or $E_2$ \eqref{entropy2}, $r_{2,\infty}$, and stationary histogram from stochastic simulations of the particle system. (c) Evolution of the entropy $E_1$  and $E_{1} - \gamma_1$ using the solutions from the AGF, the GF, and stochastic simulations. (d) Evolution of the relative entropy $E^*$ given in \eqref{bregman0} in the AGF using $E_1$ and the GF using $E_1$ and $E_2$. Parameters used given in \eqref{parameters}.}
\label{fig:case1}
  \end{center}
\end{figure}

Figure \ref{fig:case1} illustrates the behavior of solutions in the case of a convex potential $b(x) = 0.3(4x^2 + 3)$, for the parameters
\begin{align} \label{parameters}
N_r = 100, ~N_b = 500, \epsilon_r = 0.01 \text{ and } \epsilon_b = 0.015.
\end{align}
With these parameters, the volume fraction occupied by particles is 0.1. In Figure \ref{fig:case1}(a)  we plot the time evolution of \eqref{pme_d} (colored thin lines) starting from the uniform distribution. The distribution of obstacles $b(x)$ is shown as a thick black line. The black circles correspond to the histogram of obstacles obtained from 20000 samples. In Figure \ref{fig:case1}(b) we plot the stationary solution $r_*$ of \eqref{pme_d} and the equilibrium solution $r_\infty$ of the associated GF \eqref{e:gf} with the entropies $E_1$  \eqref{entropy} and $E_2$  \eqref{entropy2}. We observe that the second entropy provides a closer approximation to the stationary solution of the original equation. Additionally, we plot the histogram of the stationary distribution of the microscopic system. We find that the original equation \eqref{pme_d} provides the best approximation to the stochastic simulations. 

Next we consider the evolution of the entropy. We consider the first entropy $E_1(t)$ along the solutions of \eqref{pme_d} and \eqref{e:gf}. The entropy decays monotonically along the solutions of the GF but it does not in the case of the AGF (see Figure \ref{fig:case1}(c)). Interestingly, the entropy $E_1$ computed from the histograms of the time-dependent stochastic simulations does not decay monotonically either. In the AGF case, this is to be expected since $r_*$ is not a minimizer of the entropy functional. Nevertheless, it turns out that the relative entropy $E^*$ also decays for the AGF, and to observe this in the entropy plot one needs to account for the integral term in \eqref{bregman0}. For convenience, we denote it by $\gamma = \int_\Omega u(r_*) (r(t)-r_*) \, \ud \bfx$. 
We plot $E_1 - \gamma_1$ in Figure \ref{fig:case1}(c) along the solutions of the AGF \eqref{pme_gf}  and the stochastic particle system. We now observe monotonic decay with the modified entropy functional. The relative entropy $E^*_1(t)$ along the solutions of the AGF \eqref{pme_gf} and the GF \eqref{e:gf}, and the relative entropy $E^*_2(t)$ along the GF are shown in a logarithmic plot in Figure \ref{fig:case1}(d). We observe the exponential decay expected for a GF with a similar rate in all cases (we show this in Subsection \ref{sec:exp_conv}). 

Figures \ref{fig:case1}(c-d) suggest that the AGF \eqref{pme_d} inherits properties typical of a GF, namely that the relative entropy functional $E^*$ is decaying along its solutions and that the convergence to its stationary solution $r_*$ is exponential in time. A rigorous analysis of this problem is a challenging question for future work. 
\def \scc {0.6}
\def \scl {.8}
\begin{figure}
\unitlength=1cm
\begin{center}
\vspace{3mm}
\psfrag{t}[][][\scl]{$t$}
\psfrag{E}[][][\scl]{$E$} 
\psfrag{Eagradf}[][][\scl]{$E$ AGF}
\psfrag{Egammagf}[][][\scl]{$E - \gamma$ AGF}
\psfrag{Egradf}[][][\scl]{$E$ GF2}
\psfrag{Esimul}[][][\scl]{$E$ Sim.}
\psfrag{Egamsim}[][][\scl]{$E-\gamma$ Sim.}
	\includegraphics[height = .4\textwidth]{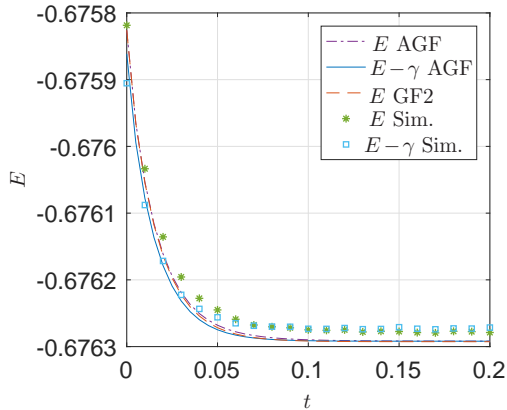}  
\caption{Evolution of the entropy $E_2(r(t))$ and $E_{2}(r(t)) - \gamma_2(r(t),r_*)$ using the solutions from the AGF, the GF, and stochastic simulations shown in Figure \ref{fig:case1}.}
\label{fig:case1_E2}
  \end{center}
\end{figure}
The fact that the first entropy $E_1$ does not decay monotonically for the stochastic simulations of the microscopic model suggests that $E_1$ is not a very good entropy for our model. The correct way to plot the entropy of the particle-based model would be to compute it directly from the simulations as a microscopic entropy, instead of using the histograms of $r(\bfx, t)$ and $b(\bfx)$ and the macroscopic entropy $E_1$. However, this is out of the scope of this paper since it would require performing a very high-dimensional density estimation to compute the entropic term of the $d N_r-$dimensional particle-based model. Instead, we examine whether the second entropy $E_2$ \eqref{entropy2}, which gave better results in Figure \ref{fig:case1}(b) for the stationary solution, is a better entropy for the microscopic system. To this end, Figure \ref{fig:case1_E2} reproduces Figure \ref{fig:case1}(c) but using $E_2$ instead of $E_1$. We see that, although there are slight differences between the entropy curves and those modified adding the extra term $\gamma_2$ (implying that their respective stationary solutions do not minimize exactly $E_2$), at least visually the second entropy decays monotonically along the solutions of our model \eqref{pme_d} and the particle-based model \eqref{ssde}. 

\def \scc {0.6}
\def \scl {.8}
\begin{figure}
\unitlength=1cm
\begin{center}
\vspace{3mm}
\psfrag{x}[][][\scl]{$x$} \psfrag{t}[][][\scl]{$t$} \psfrag{r}[][][\scl]{$ r(x,t)$} \psfrag{E}[][][\scl]{$E$} \psfrag{RelE}[][][\scl]{$E^*$} \psfrag{t}[b][][\scl]{$t$}
\psfrag{a}[][][\scl]{(a)} \psfrag{b}[][][\scl]{(b)} \psfrag{c}[][][\scl]{\quad (c)} \psfrag{d}[][][\scl]{(d)}
\psfrag{data1}[][][\scl]{Sim.}
\psfrag{equil1}[][][\scl]{$r_{1,\infty}$}
\psfrag{equil2}[][][\scl]{$r_{2,\infty}$}
\psfrag{ss}[][][\scl]{$r_*$}
\psfrag{AGF}[][][\scl]{AGF}
\psfrag{GF1}[][][\scl]{GF1}
\psfrag{GF2}[][][\scl]{GF2}
\psfrag{agf breg}[][][\scc]{AGF1-Breg.}
\psfrag{sim gf1}[][][\scc]{Sim. GF1}
\psfrag{sim breg1}[][][\scc]{Sim. GF1-Breg.}

\psfrag{Eagradf}[][][\scl]{$E$ AGF}
\psfrag{Egammagf}[][][\scl]{$E - \gamma$ AGF}
\psfrag{Egradf}[][][\scl]{$E$ GF1}
\psfrag{Esimul}[][][\scl]{$E$ Sim.}
\psfrag{Egamsim}[][][\scl]{$E-\gamma$ Sim.}
	\includegraphics[height = .4\textwidth]{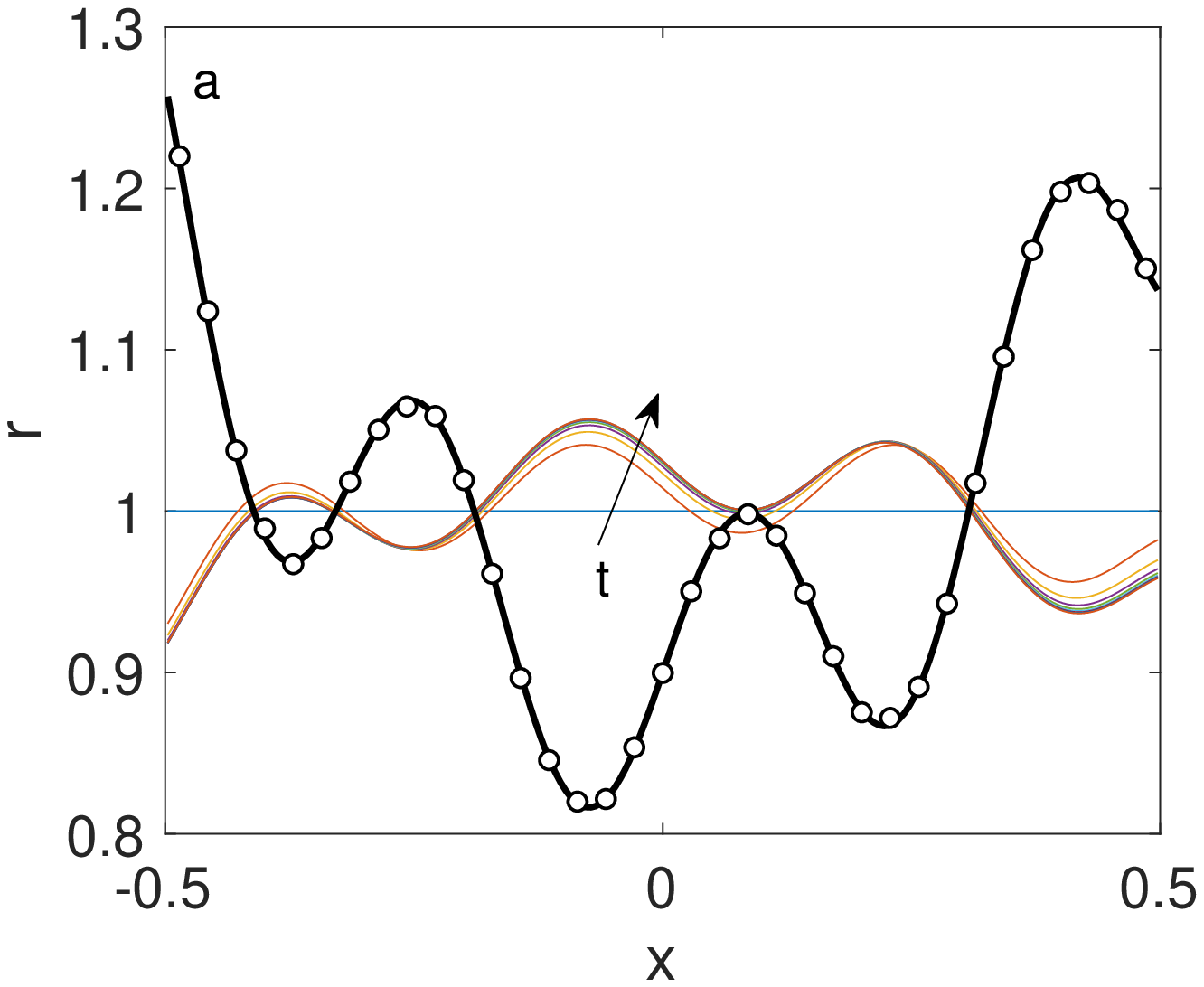} \hfill 
	\includegraphics[height = .4\textwidth]{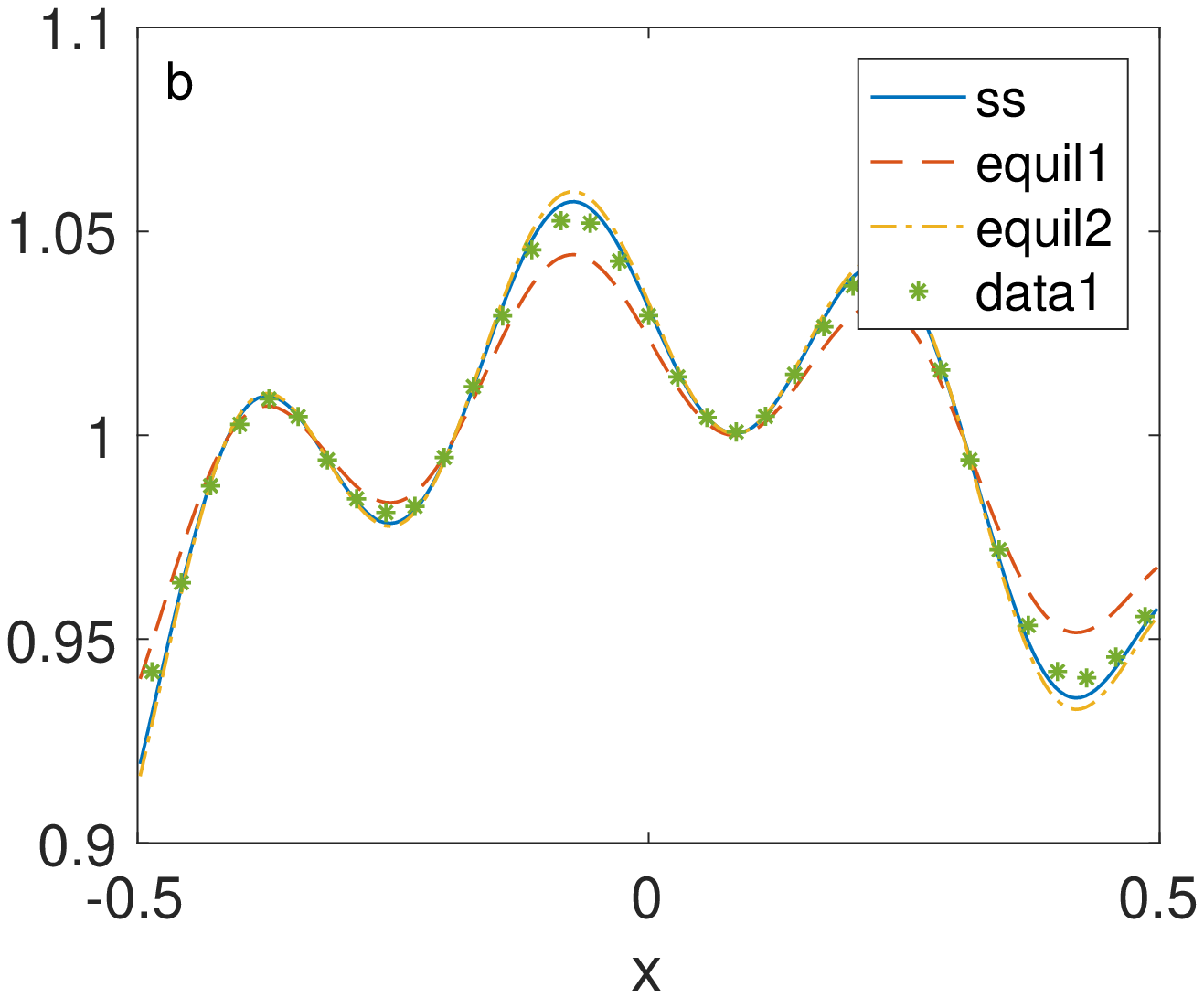}\\ \vspace{3mm}
	\includegraphics[height = .4\textwidth]{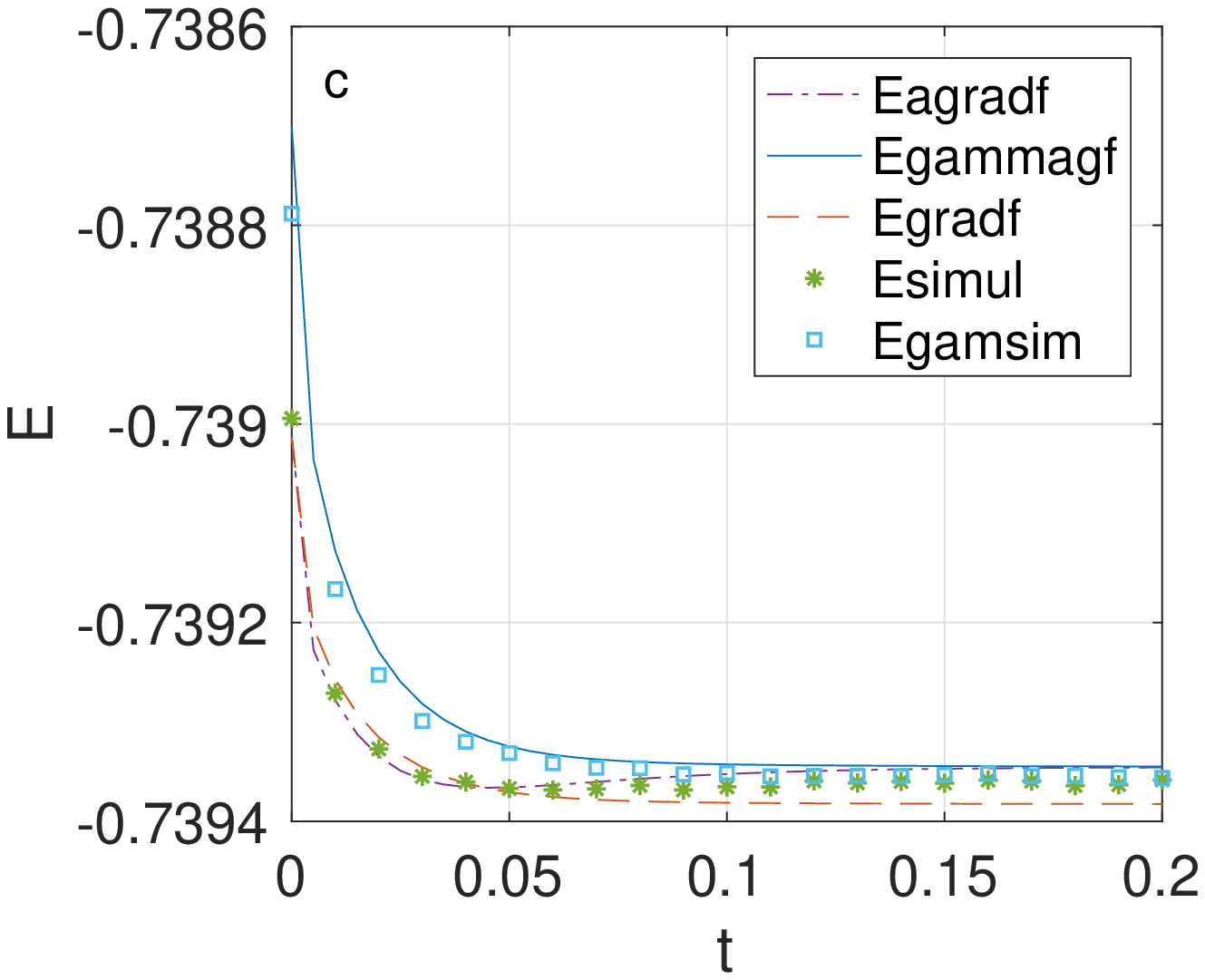} \hfill
	\includegraphics[height = .4\textwidth]{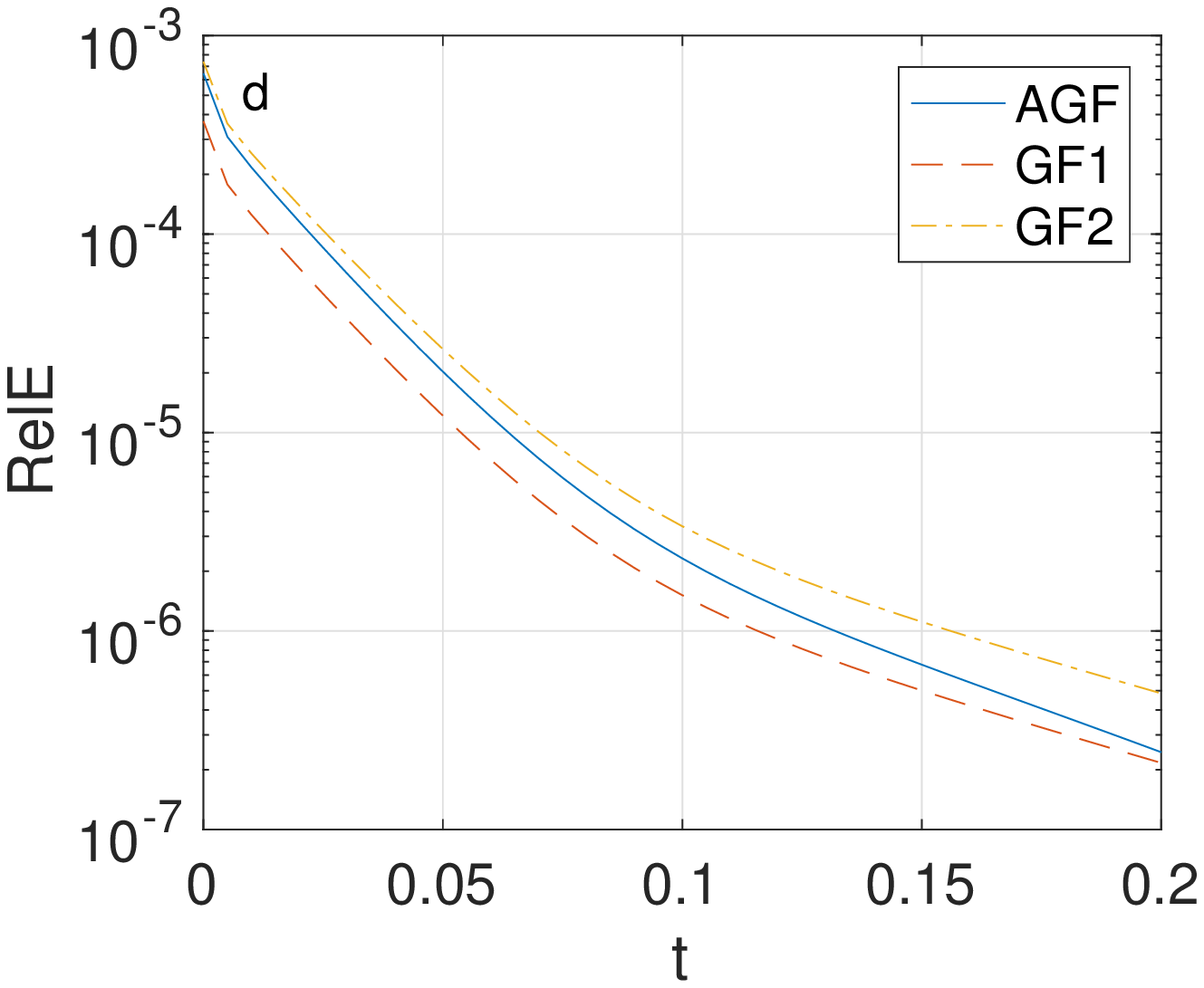} 
\caption{Example with a nonconvex porosity distribution.
Example with a convex porosity distribution. (a) Time evolution of \eqref{pme_d} with a uniform initial density (times shown: $t = 0, 0.025, 0.05, \dots, 0.2$), and obstacles distribution $b(x) = 1.2(1+0.1 \sin(20x))(x^2 + 0.75)$ (circles obtained from stochastic simulations). (b) Stationary solution $r_*$ of \eqref{pme_d}, equilibrium solution of \eqref{e:gf} with $E_1$ \eqref{entropy} or $E_2$ \eqref{entropy2}, and stationary histogram from stochastic simulations of the particle system. (c) Evolution of the entropy $E_1(t)$ using the solutions from the AGF, the GF, and stochastic simulations, as well as the Bregman-modified entropy. (d) Evolution of the relative entropy $E^*(t)$ in the AGF using $E_1$ and the GF using $E_1$ and $E_2$. Parameters used given in \eqref{parameters}.}
\label{fig:case2}
  \end{center}
\end{figure}

Next we consider a perturbation of a convex potential for $b$ as depicted in Figure \ref{fig:case2}(a). We again plot the stationary solutions of the original equation and its AGFs, and the evolution of the entropy and relative entropy. The numerical experiments indicate exponential convergence of the relative entropy functional (Figure \ref{fig:case2}(d)). In particular, the original AGF equation \eqref{pme_d} captures best the behavior of the stochastic particle system, with the GF with second entropy $E_2$ not being far off (see Figure \ref{fig:case2}(b)). We recall 
that we impose no-flux boundary conditions on a bounded domain, hence it is not surprising to still observe exponential convergence 
(we discuss this further in Subsection \ref{sec:exp_conv}).

In the examples above we have seen that the second entropy-mobility pair \eqref{pair2} provided the best approximation to the solutions of \eqref{pme_d} and the stochastic simulations of \eqref{ssde}. However, we used fixed parameters \eqref{parameters} so it is not clear if this  generalizes to other parameter values. To this end, we now examine the behavior of the error between the stationary solution $r_*$ of \eqref{pme_d} and the equilibrium solutions $r_{i, \infty}$ of the GF using each of the three entropies $E_1$ \eqref{entropy}, $E_2$ \eqref{entropy2} and $E_3$ \eqref{entropy_gen} (with $\beta =0$) while varying the  parameter $\varepsilon$ and the relative importance of self- to cross-interactions $\varepsilon_1/\varepsilon_2$. Since the three GFs are AGFs of order $\varepsilon$ to \eqref{pme_d}, we expect this error to be $O(\varepsilon^2)$. 

\def \scc {0.6}
\def \scl {.8}
\begin{figure}
\unitlength=1cm
\begin{center}
\vspace{3mm}
\psfrag{x}[][][\scl]{$x$} \psfrag{t}[][][\scl]{$t$} \psfrag{r}[][][\scl]{$ r(x,t)$} \psfrag{E}[][][\scl]{$E$} \psfrag{RelE}[][][\scl]{$DE$} \psfrag{t}[b][][\scl]{$t$}
\psfrag{a}[][][\scl]{(a)} \psfrag{b}[][][\scl]{(b)} \psfrag{c}[][][\scl]{(c)}
\psfrag{a2}[][][\scl]{(d)} \psfrag{b2}[][][\scl]{(e)} \psfrag{c2}[][][\scl]{(f)}
\psfrag{data0}[][][\scl]{$r_*$}
\psfrag{data1}[][][\scl]{GF1}
\psfrag{data2}[][][\scl]{GF2}
\psfrag{data3}[][][\scl]{GF3}
\psfrag{data4}[][][\scl]{$O(\varepsilon^2)$}
\psfrag{eps}[][][\scl]{$\varepsilon$}
\psfrag{norm}[][][\scl]{error}
	\includegraphics[height = .38\textwidth]{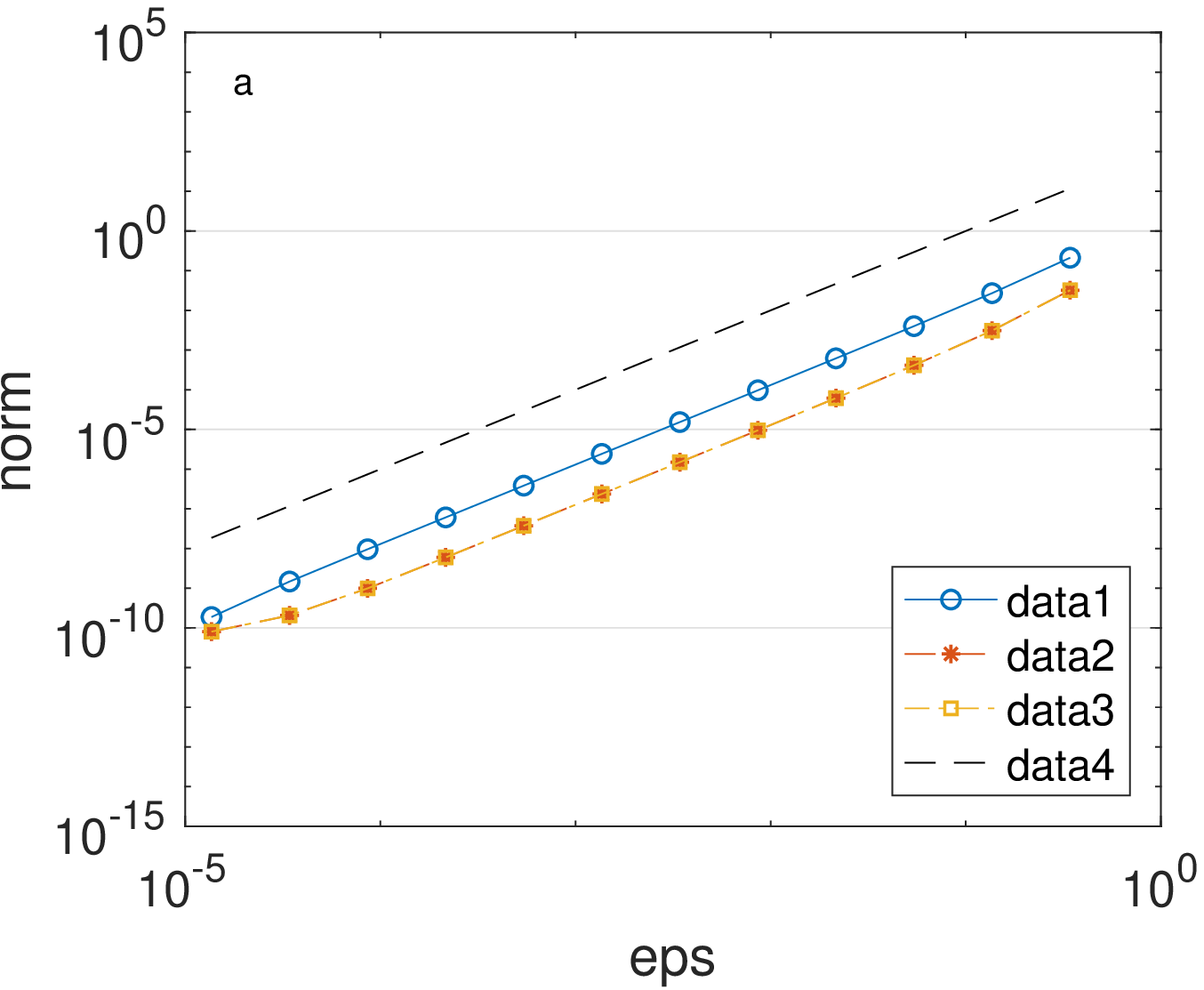} \includegraphics[height = .38\textwidth]{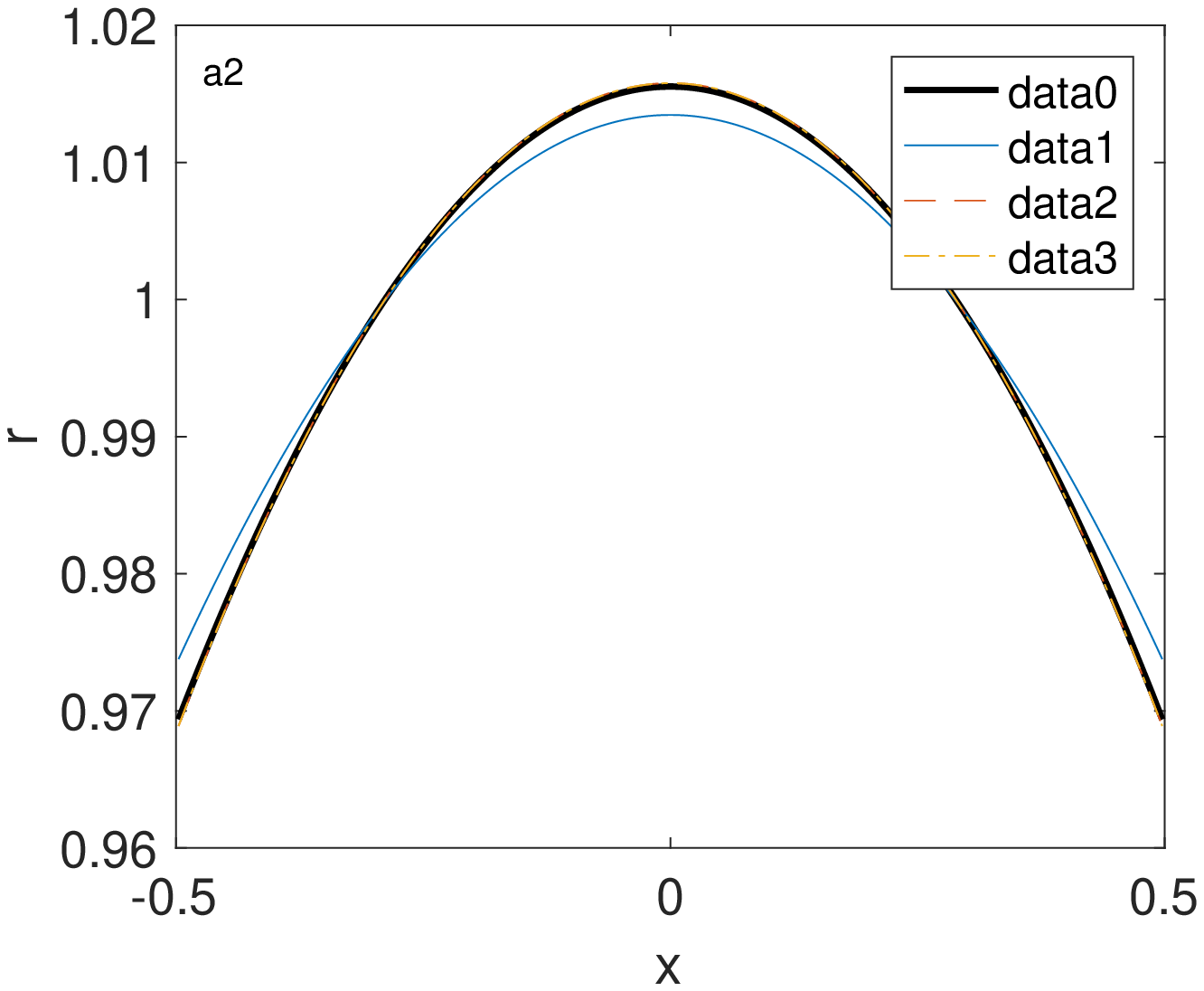} \\ 
	\includegraphics[height = .38\textwidth]{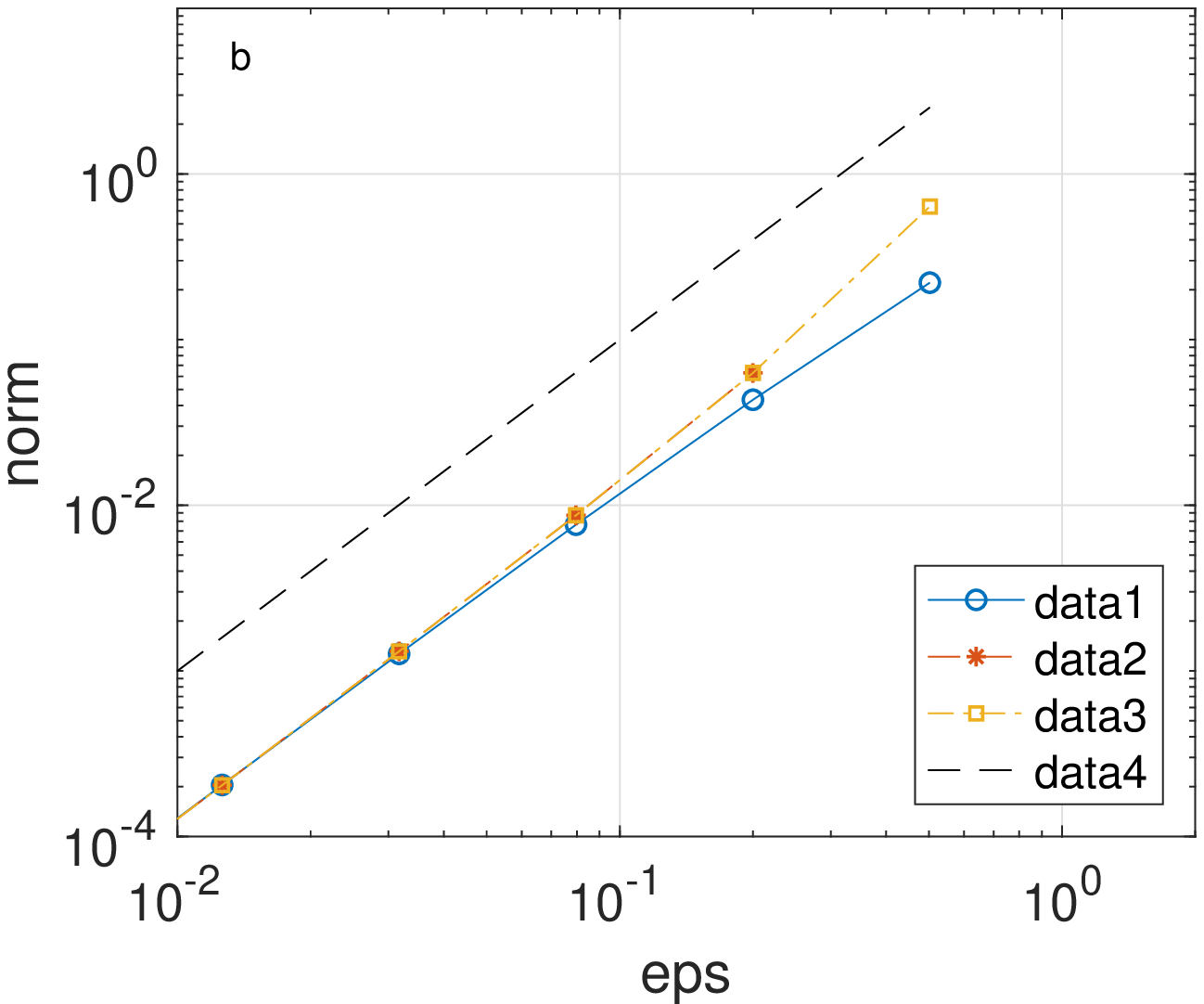} \includegraphics[height = .38\textwidth]{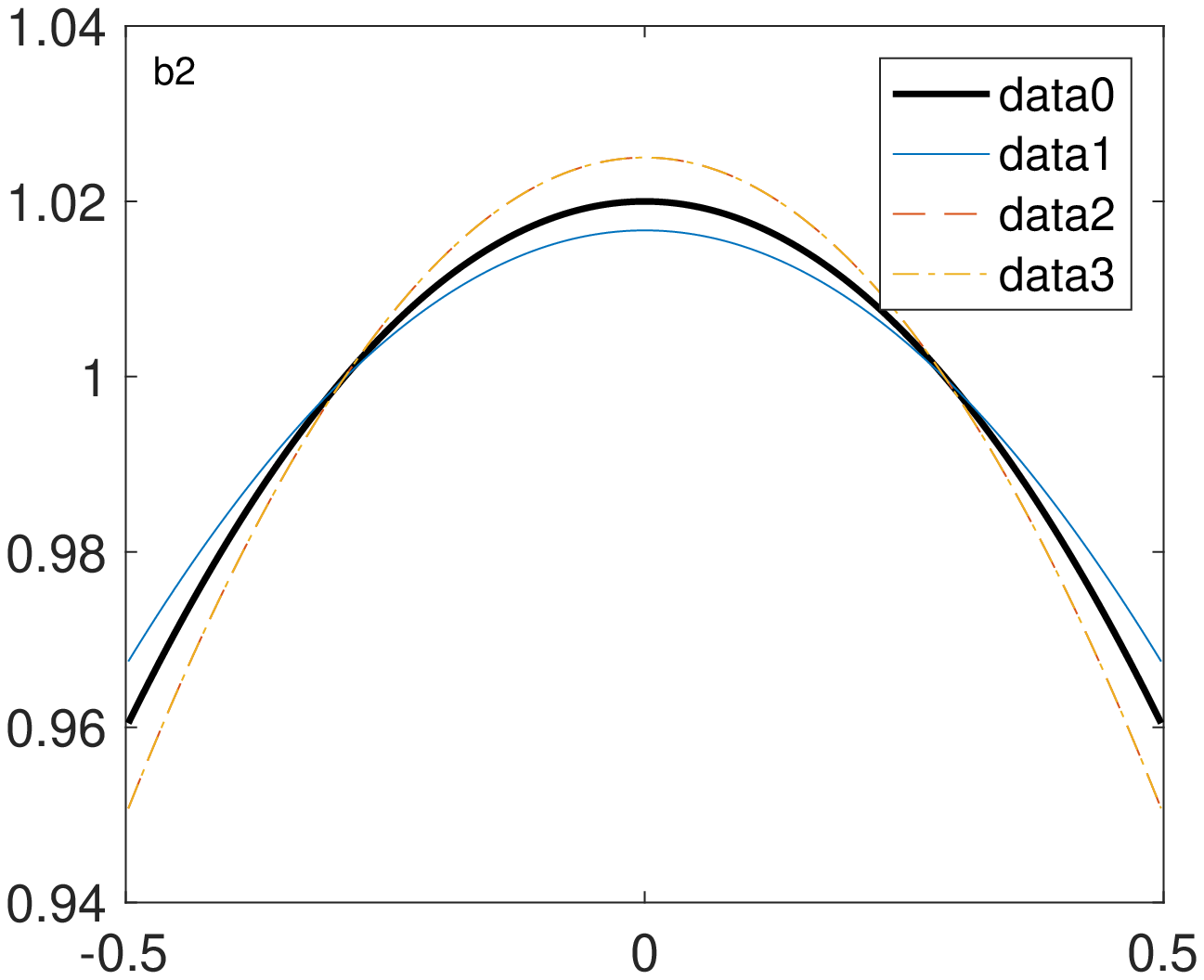} \\
	\includegraphics[height = .38\textwidth]{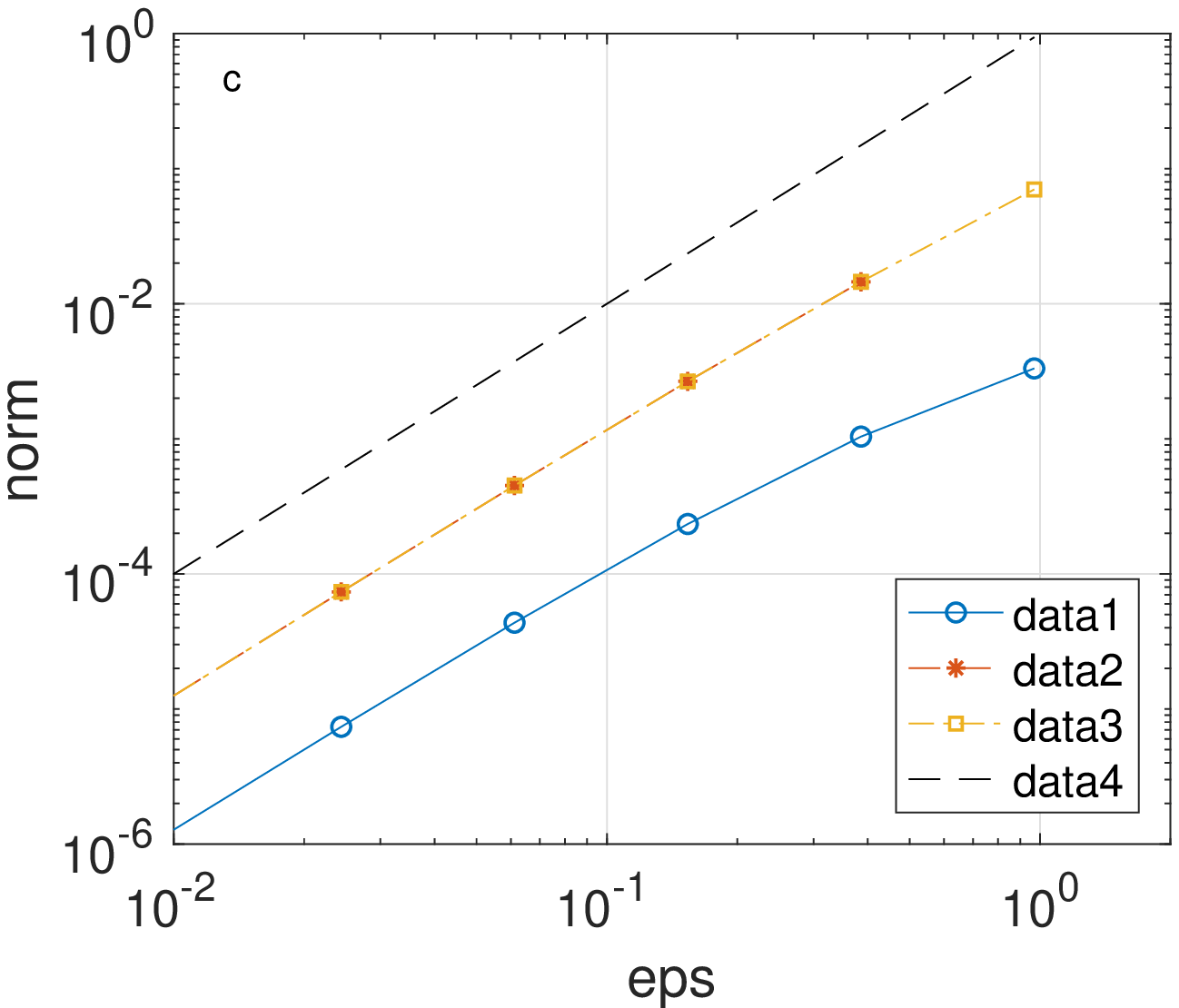} \includegraphics[height = .38\textwidth]{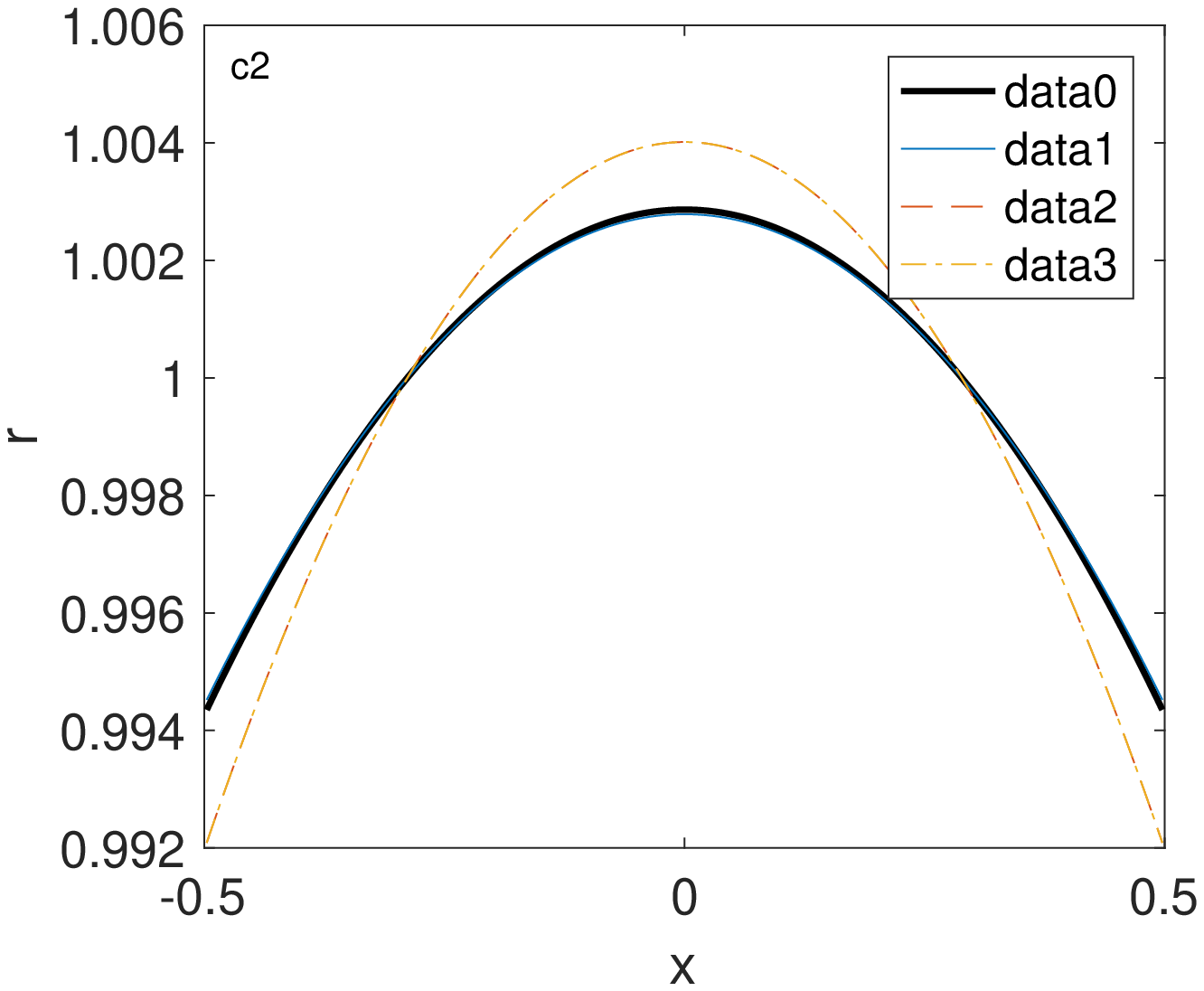}
\caption{Asymptotic behavior of the error $\| r_* - r_{i,\infty}\|$ between the stationary solution $r_*$ of the AGF and the equilibrium solution $r_{i,\infty}$ corresponding to GF i,  for $i = 1, 2, 3$ as a function of $\varepsilon$. (a) $\varepsilon_1/\varepsilon_2 = 0.1$ and $\varepsilon_2 = \varepsilon$, (b) $\varepsilon_1/\varepsilon_2 = 1$, $\varepsilon_2 = \varepsilon$, and (c) $\varepsilon_1/\varepsilon_2 = 10$ and $\varepsilon_1 = \varepsilon$. (d) Solutions corresponding to (a) for $\varepsilon=1$. (e) Solutions corresponding to (b) for $\varepsilon=1$. (f) Solutions corresponding to (c) for $\varepsilon=1$. }
\label{fig:comparison}
  \end{center}
\end{figure}

In Figure \ref{fig:comparison} we plot the $L_2$-norm of the difference between the stationary solution $r_*$ of \eqref{pme_d} and the three equilibrium solutions $r_{i, \infty}$. The three rows correspond to different ratios of the self-interaction $\varepsilon_1$ to obstacles-interaction $\varepsilon_2$ (we recall that for $d=2$,  $\varepsilon_3 = \varepsilon_2$). In all cases we observe the expected second order error in $\varepsilon$. In Figure \ref{fig:comparison}(a), $\varepsilon_1 < \varepsilon_2$ so that self-interactions are weaker than interactions with the obstacles, and we find that the second and third entropies are closest to the AGF \eqref{pme}. In Figure \ref{fig:comparison}(b) we choose $\varepsilon_1  =  \varepsilon_2$, for which the error for all three entropies is almost the same. Finally, in Figure \ref{fig:comparison}(c) $\varepsilon_1 > \varepsilon_2$ so that self-interactions are the dominant ones, and we find that the first entropy has a smallest error. The plots in the right column of Figure \ref{fig:comparison} show the solutions used to compute the point $\varepsilon=1$ in the left column. It appears as if the stationary solution of \eqref{pme_d} is bounded between the equilibrium solutions given by the entropies $E_1$ and $E_2$ for all values of $\varepsilon$. 

We have seen that if self-interactions are small compared to interactions with the obstacles ($\varepsilon_1 \ll \varepsilon_2$), GF2 approaches the AGF. This can be explained by noting that, in the limit $\varepsilon_1 \to 0$ (and $d=2$ so that $\varepsilon_2 = \varepsilon_3$), the higher-order term in the second AGF $f_2$ \eqref{f_GF2} becomes identically zero at all orders. This implies that GF2 becomes an \emph{full} GF for our model. In contrast, when self-interactions are the dominant ones ($\varepsilon_2 \ll \varepsilon_1$), the error in GF2 increases but GF1 is very close to the AGF. This again can be explained by looking at the higher-order term of the first AGF, $f_1$ in \eqref{hot_f}. When $\varepsilon_2 = \varepsilon_3 \to 0$, $f_1$ vanishes and our model becomes a full GF with the first entropy-mobility pair. 

\section{Analysis of the full gradient flow (GF) structure.} \label{sec:GF_analysis}

Before addressing the analysis of the asymptotic gradient flows in the next section, we present several analytic results for the respective full gradient flow structures. We start by reviewing the analysis for the first entropy-mobility pair \eqref{pair1}, which follows known results from the literature. In the case of the second entropy mobility pair \eqref{pair2} no such results are available. Hence we state the full proofs.  \\
In what follows we focus on the two-dimensional case $d=2$, for which the AGF \eqref{pme_d} becomes
\begin{align}\label{pme}
\partial_t r &= \nabla  \cdot \left[   ( 1 +
 \varepsilon_1   r -\varepsilon_2 b) \nabla {  r} + \varepsilon_2 r \nabla { b}  \right],   
\end{align}
with $\varepsilon_1 = (N_r-1) \pi \epsilon_r^2$ and $\varepsilon_2 = N_b \pi \epsilon_{rb}^2$. All analytic results presented in this and the next section extend in a straightforward manner to the three-dimensional problem. The AGF problem \eqref{pme}, as well as the associated GF problems, are complemented with no-flux boundary conditions on $\partial \Omega$ and initial data $r(\bfx,0)=r_0(\bfx)$. Since \eqref{pme} is valid in the small-volume fraction limit, we require its solutions to stay in the set
\begin{equation}\label{equ:set}
\mathcal{S}=\left\{r\in \mathbb{R}:r\geq 0,\varepsilon_1 r+\varepsilon_2 b \leq 1\right\}.
\end{equation}
Furthermore, we assume that $b:\overline{\Omega}\to [0,1/\varepsilon_2)$ is fixed such that
\begin{equation}
	\label{bound_b}
	0<c_1\leq 1-\varepsilon_2 b\leq 1.
\end{equation}
We note that these bounds correspond (for $d=2$ and $N_r$ large) to the following bounds on the volume concentrations, $4(\phi_r + \phi_b) \le 1$ and $4\phi_b \le 1$ (see Section \ref{sec:model}). This is then consistent with the small-volume fraction assumption to derive \eqref{pme}.

\subsection{Global in time existence and long time behavior for $(m_1,E_1)$}

In the first part of this section we discuss the analysis of the full GF structure for the first entropy mobility pair \eqref{pair1}.
This pair accounts for pairwise interactions via a quadratic term in the entropy, but does not contain any information on the correct physical bounds on the density.
However, the advantage of the first pair is that it allows us to 
prove global in time existence of the AGF \eqref{pme} (see Section \ref{sec:AGF_analysis}),
whereas for the second pair it is not clear how to control the structure of the higher-order term $f_2$ in \eqref{f_GF2}. 
In addition, we can prove exponential convergence to equilibrium using a logarithmic Sobolev inequality for $E_1$.

\subsubsection{Well-posedness and global in time existence.}
The well-posedness of the related two species model (we recall that we consider the special case of one immobile obstacle species) was 
presented in \cite{bruna2016cross}. These results can be adapted in a straight forward way to our problem 
and will not be detailed in the following. The existence and uniqueness proofs of solutions to the GF and the AGF
follow the same arguments. Hence we omit the existence proof for the GF structure and present the detailed proof for the AGF in Section \ref{sec:AGF_analysis} only.  

\subsubsection{Exponential convergence for convex potentials b.} \label{sec:exp_conv}
In this subsection, we prove exponential convergence towards equilibrium under some assumptions on the potential $b$ for the full GF 
\begin{align}\label{pme_gf_mod}
\partial_t r &= \nabla \cdot \left[m_1(r)\nabla\frac{\delta E_1}{\delta r}\right],
\end{align}
using the first entropy-mobility pair \eqref{pair1}, which we rewrite here for ease of reference:
\begin{equation*}
m_1(r)=r(1-\varepsilon_2 b) \qquad \text{and}  \qquad E_1(r)=\int_{\Omega} \left[ r(\log r-1)+\frac{1}{2}\varepsilon_1 r^2+\varepsilon_2 br\right] \ud \bfx.
\end{equation*}

Different approaches to study the long time behavior of GFs can be found in the literature, for example using the HWI method \cite{OV2000}  or the Bakry--\'Emery strategy \cite{bakry1985diffusions}.  Lisini \cite{L2009} generalized the former approach to study the equilibration behavior of scalar equations with a GF structure and nonlinear mobility. These equations can be interpreted as GFs with respect to the 2-Wasserstein distance and a Riemann distance induced by the nonlinear uniformly elliptic mobility. In this subsection we follow the Bakry--\'Emery strategy \cite{bakry1985diffusions}. This approach has been used successfully to study the equilibration behavior of various scalar PDEs, see for example \cite{carrillo2001entropy,carrillo2003kinetic, MV99}.  A key ingredient for proving exponential convergence to stationary solutions is the relative entropy functional \eqref{bregman0}, also known as the Bregman distance between the evolution and the stationary solution \cite{Burger:2016hu}. As already discussed in Section \ref{sec:model}, equation \eqref{pme_gf_mod} agrees with \eqref{pme_gf} up to order $\varepsilon$. 

We first show exponential convergence towards equilibrium for equation \eqref{pme_gf_mod} in case of a given uniformly convex potential $b \in W^{2,\infty}(\Omega)$, that is, $\text{Hess}(b) \geq \tilde{\lambda} I$ for some $\tilde{\lambda} >0$.  Let $r_{\infty}$ denote the equilibrium solution of \eqref{pme_gf_mod}. Due to the strict convexity of the entropy functional $E_1(r)$, there exists a unique and bounded minimizer $r_\infty$, that is,  $r_\infty \in L^\infty(\Omega)$.
To show exponential convergence towards $r_\infty$ we  derive a so-called logarithmic Sobolev inequality for the corresponding relative entropy functional \cite{MV99}
\begin{align}\label{logSobolev2}
E_1^*(r)= E_1(r)-E_1(r_\infty) \leq  \frac{1}{\lambda} I(r),
\end{align} 
where $\lambda >0$ and $I(r)$ denotes the dissipation of the entropy functional. Note that the last term of the relative entropy in \eqref{bregman0} drops out in the case of a GF since $u_1(r_\infty)$ is constant.  Due to the full GF structure, the entropy functional $E_1(r)$ satisfies:
\begin{align*}
\frac{\mathrm{d} E_1(r)}{\mathrm{d}t}&= -\int_\Omega m_1(r) |\nabla u_1|^2\, \ud \bfx=:-I(r),
\end{align*} 
where $I(r)$ denotes the so called dissipation term. Furthermore, we also have
\begin{align}\label{rel_ent_dissi}
\frac{\mathrm{d} E_1^*(r)}{\mathrm{d}t}&=-I(r),
\end{align}
using $u_{1,\infty} =\chi$ constant, as well as the mass conservation property. Thanks to the bounds on $b$ \eqref{bound_b}, we have 
\begin{align}\label{inequ4}
-I(r)\leq -c_1 \int_\Omega r|\nabla u|^2\, \ud \bfx=:-I_0(r).
\end{align} 
Classical results for the modified equation with a linear mobility $\partial_t r = \nabla \cdot (c_1 r \nabla u_1 )$,  give the logarithmic Sobolev inequality
\begin{align}\label{logSobolev}
0\leq E^*(r(t))&\leq \frac{1}{\lambda}I_0(r(t)),\quad t\geq 0,
\end{align} 
with $\lambda=2 c_1 \tilde{\lambda}$. Then Gronwall's lemma gives the following result:
\begin{lemma}[Exponential convergence to equilibrium]
Let $b(x)$ be a uniformly convex potential, that is, $ \emph{Hess}(b) \geq \tilde{\lambda} I$ for some $\tilde{\lambda} >0$. Then for any weak solution $r(t)$ to equation \eqref{pme_gf_mod} with $E_1^*(r_0<\infty$ we have exponential convergence to equilibrium, 
\begin{align*}
E_1^*(r(t))&\leq  E_1^*(r_0) e^{-\lambda t},
\end{align*}
where $\lambda=2 c_1 \tilde{\lambda}$.
\end{lemma}

\begin{remark}
We have shown exponential convergence to equilibrium for a uniformly convex potential $b(x)\in H^1(\Omega)$ in terms of a logarithmic Sobolev inequality. We can generalize this result to bounded perturbations of the potential $b$ using the results in \cite{AMTU2001,gross1990notes,holley1987logarithmic}.
In particular, denoting $r^b_\infty$ the equilibrium solution of equation \eqref{pme_gf_mod} with potential $b$, we obtain exponential convergence to equilibrium $r^{\tilde{b}}_{\infty}$ for the case $\tilde{b}(x)=b(x)+p(x)$ with a perturbation $p(x) \in H^1(\Omega)\cap L^\infty(\Omega)$ of zero mass, $\int_\Omega p(x)\, \ud \bfx =0$ and $0<a_1\leq r^p_\infty\leq a_2<\infty$. 
\end{remark}
\begin{remark}
As we impose no-flux boundary conditions on the domain, we expect exponential convergence also for non-convex potentials $b(\bfx) \in W^{1,\infty}(\Omega)$. This is because the no-flux boundary conditions can be interpreted as a convex potential taking the value $+\infty$ outside the domain. For example, \cite{Alasio:2017th} considers a general Fokker--Planck equation with nonlinear diffusion and nonlocal terms and shows that it is possible to recover the gradient-flow formulation of the problem in a bounded domain $\Omega$ from a sequence of problems in the whole space with a strong confining potential that becomes infinity outside $\Omega$ in the limit.
Hence, we can intuitively replace the no-flux boundary conditions by a strong confining potential that becomes infinity outside the region of interest.

\end{remark}

\subsection{Global in time existence for \texorpdfstring{$(m_2,E_2)$}{(m2,E2)}.}
Next we discuss the analysis of the full GF using the second entropy pair \eqref{pair2}. This pair has the advantage  that we automatically 
obtain the necessary bounds on $r$, which ensure that the solution stays inside the set $\mathcal S$ \eqref{equ:set}.
More specifically, the entropy functional has a structure allowing us to use the boundedness by entropy method  \cite{burger2010nonlinear,jungel2015boundedness}. We recall that the first pair \eqref{pair1} does not provide this property. 

We consider the full GF 
\begin{align}\label{pme2}
\partial_t r &= \nabla \cdot \left[m_2(r)\nabla\frac{\delta E_2}{\delta r}\right],
\end{align}
using the second entropy-mobility pair \eqref{pair2}, which we rewrite below for ease of reference:
\begin{align} \label{pair2b}
m_2(r)=r(1-\varepsilon_1 r-\varepsilon_2 b), \qquad 
E_2(r)=\int_\Omega r\left[\log\left(\frac{r}{1-\varepsilon_1 r -\varepsilon_2 b}\right)-1\right] \ud \bfx.
\end{align}	

We look for a weak solution $r: \Omega \times (0,T) \to \mathcal{S}$ that satisfies the formulation
\begin{align}\label{weaksol}
\int_0^T \! \!  \left[ \langle \partial_t r,  \Phi \rangle_{H^{-1},H^1} + \int_\Omega \left(m_2(r) \nabla \frac{\partial E_2}{\partial r} \right)\cdot \nabla \Phi\ud \bfx \right] \, \ud t&=0,
\end{align}
for all $\Phi \in L^2 (0,T, H^1(\Omega))$.

\begin{theorem}[Global existence]\label{theorem2}
Let $T>0$ and $b:\Omega \to [0,1/\varepsilon_2)$ be a given function in $H^1(\Omega)$. Furthermore let $r_0:\Omega \to \mathcal{S}^\circ$, where $\mathcal{S}$ is defined by \eqref{equ:set}, be a measurable function such that $E_2( r_0)<\infty$. Then, there exists a weak solution $r:\Omega\times (0,T)\to \mathcal{S}$ to equation \eqref{pme2} in the sense of \eqref{weaksol} such that 
\begin{align*}
\partial_t  r \in L^2(0,T;H^1(\Omega)'),\qquad 
 r \in L^2(0,T;H^1(\Omega )).
\end{align*}
Moreover, the solution satisfies the following entropy dissipation inequality:
\begin{align}
E_2(r(t))+\int_0^t \int_{\Omega} C_1|\nabla  r|^2 \,\ud \bfx\leq E_2(r(0)) + Ct,
\end{align}
for some constants $C_1>0$, $C\geq0$ depending only on $\varepsilon_1$, $\varepsilon_2$, and $\nabla b$.
\end{theorem}

The proof of Theorem \ref{theorem2} is based on an implicit in time discretization and subsequent regularization of equation \eqref{pme2}. This approach has been used for many nonlinear PDEs with a gradient flow structure, hence we state the main steps only. The proof for the AGF equation \eqref{pme_d} is based on the same idea, but requires a more detailed analysis and will be presented in Section \ref{sec:AGF_analysis}. 

To show global in time existence for the full GF, we discretize equation \eqref{pme2} using the implicit Euler scheme. More specifically, let $N\in \mathbb{N}$ and consider the discretization of $(0,T]$ into subintervals
\begin{equation}\label{time_disc}
(0,T]=\bigcup_{k=1}^N ((k-1)\tau,k\tau], \quad \text{with} \quad \tau=\frac{T}{N}.
\end{equation} Then for a given function $r_{k-1} \in \mathcal{S}$ approximating $r$ at time $\tau(k-1)$, we obtain a recursive sequence of elliptic problems, which are then regularized by higher-order terms, that is,
\begin{align}\label{pme_reg3}
\frac{ r_k- r_{k-1}}{\tau} &=\nabla\cdot\left[  m_2(r_k) \nabla u_k \right ] +\tau( \Delta u_k-u_k),
\end{align} 
where $u_k$ is the discretized entropy variable $u_2$ in \eqref{ent_var2}.
The regularization terms are needed to show existence of weak solutions to a linearized version of the problem \eqref{pme_reg3} using Lax--Milgram. The existence of the corresponding nonlinear problem is obtained by applying Schauder's fixed point theorem. The following Lemma guarantees that our solution lies in the correct set using the boundedness by entropy method, \cite{jungel2015boundedness}. 

\begin{lemma}\label{lemma6}
The entropy density 
\begin{align} \label{pair2density}
h:\mathcal{S}^\circ\to \mathbb{R}, \qquad r \mapsto r\left[\log \left( \frac{r}{1-\varepsilon_1 r-\varepsilon_2 b}\right)-1\right],
\end{align}
is strictly convex and belongs to $C^2(\mathcal{S}^\circ).$ Its derivative $h':\mathcal{S}^\circ\to \mathbb{R}$ is invertible and the inverse of second derivative $h'':\mathcal{S}^\circ\to \mathbb{R}$ is uniformly bounded.
\end{lemma}

\begin{proof}
We have that
\begin{equation*}
h'= \log  r -\log (1-\varepsilon_1 r-\varepsilon_2 b)+\frac{\varepsilon_1 r}{1-\varepsilon_1 r-\varepsilon_2 b},
\end{equation*}
and
\begin{equation*}
h''=\frac{1}{r}+\frac{2\varepsilon_1}{1-\varepsilon_1 r-\varepsilon_2 b}+\frac{\varepsilon_1^2 r}{(1-\varepsilon_1 r-\varepsilon_2 b)^2}.
\end{equation*}
As $h''$ is positive on the set $\mathcal{S}^\circ$, $h$ is strictly convex. We can easily deduce that the inverse of $h''$ is bounded on $\mathcal{S}^\circ$. Since $b$ is fixed, $h'$ is a one to one mapping from $]0,(1-\varepsilon_2 b)/\varepsilon_1[$ to $\mathbb{R}$  for every $x\in \Omega$. Hence $h'$ is invertible.
\end{proof}

Lemma \ref{lemma6} allows us to employ the transformation from the entropy variable $u_2\in \mathbb{R}$ back to the original variable $r$ and ensures that $r\in \mathcal{S}^\circ$. 

In order to provide enough regularity for passing to the limit $\tau\to 0$, we show an entropy dissipation relation for the weak formulation of the time discretization \eqref{pme_reg3}, that is,
\begin{align}
\label{pme_reg4}
\frac{1}{\tau}\int_\Omega( r_k- r_{k-1}) \Phi \, \ud \bfx &+\int_\Omega \nabla \Phi^T\, m_2(r_k) \nabla u_k \ud \bfx+\tau R(\Phi,u_k)=0,
\end{align} 
for $\Phi\in H^1(\Omega)$, where $r_k= h'^{-1}(u_k)$ and
\begin{equation*}
R(\Phi,u_k)=\int_{\Omega} \left( \Phi u_k+\nabla \Phi\cdot \nabla u_k\right)\ud \bfx.
\end{equation*}

As the entropy density $h$ is convex, we have $h(\varphi_1)-h(\varphi_2)\leq h'(\varphi_1)\cdot(\varphi_1-\varphi_2)$ for all $\varphi_1,\varphi_2\in\mathcal{S}$. If we choose $\varphi_1=r_k$ and $\varphi_2= r_{k-1}$ and using $h'( r_k)= u_k$, we obtain
\begin{align}\label{inequ5} 
\frac{1}{\tau}\int_{\Omega}&
 (r_k- r_{k-1})
 u_k\,\ud \bfx\geq\frac{1}{\tau}\int_{\Omega} \left[
h( r_k)-h( r_{k-1})
\right] \ud \bfx.
\end{align}
Inserting \eqref{inequ5} in equation \eqref{pme_reg4} with the test function $\Phi=u_k$ leads to
\begin{align}
\label{inequ6}
\int_{\Omega}h( r_k)\,\ud \bfx+\tau \! \! \int_{\Omega} \nabla u_k^T\, m_2(r_k) \nabla u_k  \ud \bfx+\tau ^2 R(u_k,u_k)\leq \int_{\Omega}h( r_{k-1})\,\ud \bfx.
\end{align}

\begin{lemma}\label{lemma7}
Let $ b :\Omega \rightarrow [0,1/\varepsilon_2)$ be such that $\nabla b \in L^2 (\Omega)$ and let $r \in L^2(\Omega), r\in \mathcal{S}^\circ$ a.e. be such that $u=h'(r)\in H^1(\Omega)$. Then, there exist two constants $C_1>0$, $C \geq 0$ depending only on $\varepsilon_1, \varepsilon_2$ and $\| \nabla b\|_{L^2 (\Omega)}$ such that 
\begin{align*}
\int_\Omega \nabla u_2^T m_2(r) \nabla u_2 \, \ud \bfx &\geq \int_\Omega C_1|\nabla \sqrt{ r}|^2\,\ud \bfx - C.
\end{align*}
\end{lemma}

\begin{proof}
Using the definition of $u_2$ \eqref{ent_var2}, we have
\begin{align*}
\nabla u_2 &=\frac{\nabla r}{r}+\frac{\varepsilon_1 \nabla r+\varepsilon_2 \nabla b}{1-\varepsilon_1 r-\varepsilon_2 b}+\frac{\varepsilon_1 \nabla r}{1-\varepsilon_1 r-\varepsilon_2 b}+\frac{\varepsilon_1 r (\varepsilon_1 \nabla r+\varepsilon_2 \nabla b)}{(1-\varepsilon_1 r-\varepsilon_2 b)^2}\\
&=\left(\frac{1}{r}+\frac{2\varepsilon_1}{1-\varepsilon_1 r-\varepsilon_2 b}+\frac{\varepsilon_1^2 r}{(1-\varepsilon_1 r-\varepsilon_2 b)^2}\right) \nabla r + \left(\frac{\varepsilon_2}{1-\varepsilon_1 r-\varepsilon_2 b}+\frac{\varepsilon_1 \varepsilon_2 r}{(1-\varepsilon_1 r-\varepsilon_2 b)^2}\right) \nabla b\\
&=\frac{1-2\varepsilon_2b+\varepsilon_2^2 b^2}{r(1-\varepsilon_1 r-\varepsilon_2 b)^2}\nabla r +\frac{\varepsilon_2 (1-\varepsilon_2 b)}{(1-\varepsilon_1 r-\varepsilon_2 b)^2}\nabla b\\
&=\frac{(1-\varepsilon_2 b)}{(1-\varepsilon_1 r - \varepsilon_2 b)^2}\left((1-\varepsilon_2b) \frac{\nabla r}{r}+\varepsilon_2 \nabla b \right).
\end{align*}
As $r\in \mathcal{S}^\circ$ a.e., that is, $0<1-\varepsilon_1 r-\varepsilon_2 b <1$ a.e., it holds that 
\begin{equation*}
\frac{1}{(1-\varepsilon_1 r-\varepsilon_2 b)^2}\geq \frac{1}{1-\varepsilon_1 r-\varepsilon_2 b} \quad \text{a.e.},
\end{equation*}
which leads to
\begin{align*}
\int_\Omega \nabla u_2^T m_2(r) \nabla u_2 \, \ud\bfx &\geq \int_{\Omega}  r (1-\varepsilon_2 b)^2\left|(1-\varepsilon_2b) \frac{\nabla r}{r}+\varepsilon_2\nabla b \right|^2\,\ud \bfx.
\end{align*}
Using Young's inequality to estimate the mixed term, and $1-\varepsilon_2b >0$, $r\in \mathcal{S}^\circ$ a.e., and $\nabla b \in L^2(\Omega)$, leads to 
\begin{align*}
\int_\Omega \nabla u_2^T m_2(r) \nabla u_2\,  \ud\bfx \geq \int_{\Omega} C_1 |\nabla \sqrt{r}|^2\ud\bfx- C,
\end{align*}
for some constants $C_1>0$ and $C \geq 0$, as required.
\end{proof}

Finally, Lemma \eqref{lemma8} states the uniform \emph{a priori} estimates in the discrete time step $\tau$ arising from inequality \eqref{inequ6} and Lemma \ref{lemma7}.

\begin{lemma}[\emph{a priori} estimates]\label{lemma8}
There exists a constant $K\in\mathbb{R}^+$ (independent of $\tau$), such that the following bounds hold:
\begin{align*}
\| \sqrt{r}_\tau\|_{L^2(0,T;H^1(\Omega))}\leq K,  \qquad 
\sqrt{\tau}\|u_\tau\|_{L^2(0,T;H^1(\Omega))}\leq K,
\end{align*}
where $ r_\tau({\bf x},t)= r_k({\bf x})$ for ${\bf x}\in\Omega$ and $t\in ((k-1)\tau,k\tau]$ and $u_\tau$ the corresponding entropy variable. 
\end{lemma}

This together with the use of Aubin--Lions lemma allow to pass to the limit $\tau \to 0$, cf. Section \ref{sec:AGF_analysis} for more details.

\section{Global existence of the asymptotic gradient flow (AGF) structure.} \label{sec:AGF_analysis}

Finally we are addressing the question of global in time existence for the AGF \eqref{pme}. We shall see that we can prove
existence and uniqueness of solutions using the first entropy-mobility pair  \eqref{pair1} only. This choice may seem a bit inept, since the second entropy-mobility pair 
provides a better approximation of the stationary solutions to the AGF and the microscopic simulations. It also gives useful bounds on the
particle density $r$ for the respective GF structure, which ensure that solutions stay in the set $\mathcal{S}$. However its stronger nonlinear nature results in more complex higher-order terms, 
which we are not able to control at the moment. Our goal is to obtain a-prori estimates from the dissipation of the entropy, which cannot be uniform however in the higher-order terms.

More specifically, as 
\begin{align*}
\nabla u_2&= \frac{\nabla r}{r}+\frac{2\varepsilon_1 \nabla r +\varepsilon_2\nabla b}{1-\varepsilon_1 r-\varepsilon_2 b}+\frac{\varepsilon_1^2 r \nabla r+\varepsilon_1 \varepsilon_2 r\nabla b}{(1-\varepsilon_1 r-\varepsilon_2 b)^2}\\
&=\frac{(1-\varepsilon_2 b)}{(1-\varepsilon_1 r - \varepsilon_2 b)^2}\left((1-\varepsilon_2b) \frac{\nabla r}{r}+\varepsilon_2 \nabla b \right),
\end{align*}

equation \eqref{pme} can be written as

\begin{align}\label{equ:AGF_expl1}
\partial_t r&=\nabla \cdot \left[ m_2(r)\nabla u_2 -\frac{\varepsilon_1^2r^2}{1-\varepsilon_1r-\varepsilon_2b}\nabla r-\frac{\varepsilon_1\varepsilon_2r^2}{1-\varepsilon_2r-\varepsilon_2 b}\nabla b\right].
\end{align}

Formally, calculating the time derivative of the entropy functional $E_2$ gives

\begin{align}
\frac{\mathrm{d} E_2(r)}{\mathrm{d} t}&=-\int_\Omega \left[ m_2(r)\nabla u_2 -\frac{\varepsilon_1^2r^2}{1-\varepsilon_1r-\varepsilon_2b}\nabla r-\frac{\varepsilon_1\varepsilon_2r^2}{1-\varepsilon_2r-\varepsilon_2 b}\nabla b\right]\nabla u_2\, \ud\bfx\\
&=- \int_\Omega \frac{1-\varepsilon_2b}{(1-\varepsilon_1 r-\varepsilon_2 b)^2}\left[\frac{(1-\varepsilon_2 b)((1-\varepsilon_2b)^2-\varepsilon_1^2r^2)}{r(1-\varepsilon_1 r-\varepsilon_2b)}| \nabla r| ^2+r\varepsilon_2^2 | \nabla b| ^2\right.\\
&\quad \qquad \left.+(1-\varepsilon_2b)\varepsilon_2 \nabla r\nabla b+\frac{\varepsilon_2((1-\varepsilon_2b)^2-\varepsilon_1^2r^2)}{1-\varepsilon_1 r-\varepsilon_2 b}\nabla r \nabla b\right] \ud\bfx
\end{align}

It is not clear how to control the mixed terms especially as soon as the particle concentration reaches 
the maximum density, that is, if $\varepsilon_1 r +\varepsilon_2b \approx 1$, for which the last term can become dominant and arguments based on for example the Cauchy-Schwarz inequality do not yield a suitable estimate.

As a consequence we use the first entropy-mobility pair \eqref{pair1} in the proof of the global in time existence result Theorem \ref{theorem1}. Since it does not
provide any information about the volume constraint, we need to introduce a pre-factor in the flux term of \eqref{pme} and regularize the entropy to ensure that the solutions 
satisfy the maximum volume fraction constraint \eqref{equ:set}, that is, when $\varepsilon_1 r +\varepsilon_2 b = 1$. This pre-factor can be seen as a Lagrange multiplier for the constraint $r \in \mathcal S$.
We look for a weak solution $r:\Omega\times(0,T) \rightarrow \mathcal{S}$ to the equation
\begin{align} \nonumber
\partial_t r &= \nabla \cdot J_r\\
\label{pme_mod1}
(1-\varepsilon_1 r -\varepsilon_2 b)J_r&=(1-\varepsilon_1 r -\varepsilon_2 b)\left[   ( 1 +
 \varepsilon_1   r -\varepsilon_2 b) \nabla {  r} + \varepsilon_2 r \nabla { b} \big) \right],
\end{align}
in the sense of 
\begin{align}\label{weak_sol}
\int_0^T  \left[ \langle \partial_t r,  \Phi \rangle_{H^{-1},H^1} + \int_\Omega J_r\cdot \nabla \Phi \, \ud \bfx \right] \,\ud t &=0,
\end{align}
for all $\Phi \in L^2 (0,T, H^1(\Omega))$. In addition to the definition above, we will only consider weak solutions satisfying the entropy relation
\begin{align}\label{theorem1_2}
E_1(r(t))+\int_0^t \int_{\Omega} \frac{\varepsilon_1}{4}|\nabla  r|^2 \,\ud \bfx\leq E_1(r_0) + Ct,
\end{align}
for some constant $C\geq0$ depending only on $\varepsilon_1$, $\varepsilon_2$, and $\nabla b$.

\begin{theorem}[Global existence in the case of small volume fraction]\label{theorem1}
Let $T>0$, $b:\Omega \to [0,1/\varepsilon_2)$ a given function in $H^1(\Omega)$ and $E_1(r)$ be the entropy functional defined in \eqref{entropy}. Furthermore let $r_0:\Omega \to \mathcal{S}^\circ$, where $\mathcal{S}$ is defined by \eqref{equ:set}, be a measurable function such that $E_1( r_0)<\infty$. Then there exists a weak solution $r:\Omega\times (0,T)\to \mathcal{S}$ to equation \eqref{pme_mod1} in the sense of \eqref{weak_sol} and \eqref{theorem1_2} satisfying 
	\begin{align*}
\partial_t  r \in L^2(0,T;H^1(\Omega)'),\qquad 
 r \in L^2(0,T;H^1(\Omega )).
\end{align*}
\end{theorem}

To prove Theorem \ref{theorem1}, it is convenient to rewrite \eqref{pme} as
\begin{align}\label{pme_gf2}
\partial_t r &= \nabla \cdot \left[n(r)\nabla\frac{\delta E_1}{\delta r}+g(r)\right],
\end{align}
where $E$ is given in \eqref{entropy}, 
\begin{equation*}
n(r)=r(1-\varepsilon_2 b)+\frac{\varepsilon_1 \varepsilon_2 r^2b}{1+\varepsilon_1 r},
\end{equation*}
and 
\begin{equation*} 
g(r)= \frac{\varepsilon_2 ^2 rb}{1+\varepsilon_1 r}\nabla b.
\end{equation*}
We note that \eqref{pme_gf2} is not in AGF form, since the second term of $n(r)$ is of order $\varepsilon^2$. However, this structure is more convenient for the analysis since $g$ does not depend on any derivatives of $r$, and it does not change the final result (as it is only a reformulation of the original equation \eqref{pme_d}). Note also that $n(r)\geq 0$ for $r\in \mathcal{S}$. 

The proof is based on the following approximation argument. We discretize equation \eqref{pme_gf2} in time using the implicit Euler scheme with time step $\tau>0$. This gives us a recursive sequence of elliptic problems, which we regularize to obtain sufficiently smooth solutions. More specifically, we perform the time discretization \eqref{time_disc}. Then for a given function $ r_{k-1} \in \mathcal{S}$, which approximates $r$ at time $\tau(k-1)$, we want to find $r_k\in \mathcal{S}$ and the associated entropy variable $\tilde{u}_k$ solving the regularized time discrete problem
\begin{align}
\label{pme_reg}
\frac{ r_k- r_{k-1}}{\tau} =\nabla\cdot\left[ n( r_k)
\nabla \tilde{u}_k+g(r_k) \right]+\tau( \Delta \tilde{u}_k-\tilde{u}_k).
\end{align} 
Here we use the modified entropy $\tilde{E} =E_1+E_\tau$, with 
\begin{align*}
	E_\tau =   \tau  \int_{\Omega}   (1-\varepsilon_1 r-\varepsilon_2 b) [\log (1-\varepsilon_1 r-\varepsilon_2 b)-1] \ud \bfx,
\end{align*}
and the associated entropy variable
\begin{align} \label{u_reg}
\tilde{u}=u_1 + u_\tau &=  \log  r + \varepsilon_1 r +\varepsilon_2 b -\tau \varepsilon_1 \log (1-\varepsilon_1 r-\varepsilon_2 b).
\end{align}

The additional term in the entropy provides an upper bound on the solution and the higher-order regularization term guarantees coercivity of the elliptic equation in $H^1(\Omega)$. This guarantees existence of weak solutions to a linearized version of equation \eqref{pme_reg} using Lax--Milgram. Then existence of solutions to the corresponding nonlinear problem follows from Schauder's fixed point theorem  \cite{zeidler1994nonlinear}. Finally uniform \emph{a priori} estimates in $\tau$ and the use of Aubin--Lions lemma (see for example \cite{zamponi2015analysis}) allow to pass to the limit $\tau \to 0$ leading to a weak solution of \eqref{pme_mod1}. 

\begin{lemma}\label{lemma1}
The entropy density 
\begin{align*}
\tilde{h}: ~ &\mathcal{S}^\circ\to \mathbb{R}, \\
&r \mapsto r(\log   r-1) + \frac{\varepsilon_1}{2} r^2 + \varepsilon_2  r  b  + \tau(1-\varepsilon_1 r-\varepsilon_2 b)[\log(1-\varepsilon_1 r-\varepsilon_2 b)-1],
\end{align*}
is strictly convex and belongs to $C^2(\mathcal{S}^\circ).$ Its derivative $\tilde{h}':\mathcal{S}^\circ\to \mathbb{R}$ is invertible and the inverse of the second derivative $\tilde{h}'':\mathcal{S}^\circ\to \mathbb{R}$ is uniformly bounded.
\end{lemma}

\begin{proof}
We have that
\begin{equation*}
\tilde{h}'= \log  r + \varepsilon_1 r +\varepsilon_2 b -\tau \varepsilon_1 \log (1-\varepsilon_1 r-\varepsilon_2 b),
\end{equation*}
and
\begin{equation*}
\tilde{h}''=\frac{1}{ r}+\varepsilon_1+\tau\frac{\varepsilon_1^2}{1-\varepsilon_1 r-\varepsilon_2 b}.
\end{equation*}
Since $\tilde{h}''$ is positive on the set $\mathcal{S}^\circ$, $\tilde{h}$ is strictly convex. The boundedness of $h''$ and the invertibility of $h'$ follow from the same arguments as in Lemma \ref{lemma6}.
\end{proof}

\subsection{Time discretization and regularization of equation \texorpdfstring{\eqref{pme_gf2}}{20}.}
The weak formulation of equation \eqref{pme_reg} is given by: 
\begin{align}
\label{pme_reg2}
\frac{1}{\tau}\int_\Omega( r_k- r_{k-1}) \Phi \, \ud \bfx +\int_\Omega  \left[ \nabla \Phi^T\, n(r_k) \nabla \tilde{u}_k +g(r_k) \nabla \Phi\right] \ud \bfx+\tau R(\Phi,\tilde{u}_k)=0,
\end{align} 
for $\Phi\in H^1(\Omega)$, where $r_k=\tilde h'^{-1}(\tilde{u}_k)$ and
\begin{equation*}
R(\Phi,\tilde{u}_k)=\int_{\Omega} \left( \Phi\tilde{u}_k+\nabla \Phi\cdot \nabla \tilde{u}_k\right)\ud \bfx.
\end{equation*}

\begin{lemma}
Let $r_{k-1} \in L^2(\Omega)$ and $\tau > 0$. Then there exists a weak solution of \eqref{pme_reg2}, given as $r_k = \tilde h'^{-1}(\tilde{u}_k) \in L^2(\Omega)$ with $\tilde{u}_k \in H^1(\Omega)$ .
\end{lemma}
\begin{proof}
In order to perform a fixed point argument
we define $U:\mathcal{S}\subseteq L^2(\Omega)\to\mathcal{S}\subseteq L^2(\Omega), \tilde{r} \mapsto r=\tilde h'^{-1}(\tilde{u})$, where
$\tilde{u}$ is the unique solution in $H^1(\Omega)$ to the linear problem
\begin{equation} \label{equ_lax} 
a(\tilde{u},\Phi)=F(\Phi) \qquad \text{for all} \qquad \Phi \in H^1(\Omega),
\end{equation}
with 
\begin{align*}
a(\tilde{u},\Phi)&=\int_{\Omega}\nabla \Phi^T\, n(\tilde{r}) \nabla \tilde{u}\,\ud \bfx +\tau R(\Phi,\tilde{u}),\\
F(\Phi)&=-\frac{1}{\tau}\int_\Omega \left[( \tilde{r}- r_{k-1}) \Phi -g(\tilde{r}) \nabla \Phi\right] \ud \bfx.
\end{align*}
The bilinear form $a:H^1(\Omega)\times H^1(\Omega)\to \mathbb{R}$ and the linear functional $F:H^1(\Omega)\to \mathbb{R}$ are bounded. Moreover, $a$ is coercive since the positivity of $n(r)$ implies that
\begin{align*}
a(\tilde{u},\tilde{u}) =\int_{\Omega}\nabla \tilde{u}^T\, n(\tilde{r}) \nabla \tilde{u}\,\ud \bfx +\tau R(\tilde{u},\tilde{u})\geq \tau\|\tilde{u}\|_{H^1(\Omega)}^2.
\end{align*}
Then the lemma of Lax--Milgram guarantees the existence of a unique solution $\tilde{u}\in H^1(\Omega)$ to \eqref{equ_lax}.

To apply Schauder's fixed point theorem, we need to show that the map $U$:
\begin{enumerate}[(1)]
\item  sends a convex, closed set onto itself,
\item  is compact,
\item  is continuous.
\end{enumerate}
For the continuity, let $\tilde{r}_k$ be a sequence in $\mathcal{S}$ converging strongly to $\tilde{r}$ in $L^2(\Omega)$ and let $\tilde{u}_k$ be the corresponding unique solution to \eqref{equ_lax} in $H^1(\Omega)$. Due to the structure of $n$, we also have that $n(\tilde{r}_k)\to n(\tilde{r})$ strongly in $L^2(\Omega)$. The positivity of $n$ for $ r\in \mathcal{S}$ provides a uniform bound for $\tilde{u}_k$ in $H^1(\Omega)$. Hence, there exists a subsequence with $\tilde{u}_k\rightharpoonup \tilde{u}$ weakly in $H^1(\Omega)$. The $L^{\infty}$ bounds of $n(\tilde{r}_k)$ and the application of a density argument allow us to pass from test functions $\Phi\in W^{1,\infty}(\Omega)$ to test functions $\Phi\in H^1(\Omega)$. So, the limit $\tilde{u}$ as the solution of problem \eqref{equ_lax} with coefficients $\tilde{r}$ is well defined. 
Due to the compact embedding $H^1(\Omega)\hookrightarrow L^2(\Omega)$, we have a strongly converging subsequence of $\tilde{u}_k$ in $L^2(\Omega)$. Since the limit is unique, the whole sequence converges. From Lemma \ref{lemma1} we know that $ r=h'^{-1}(\tilde{u})$ is Lipschitz continuous, which yields continuity of $U$.
Since $\mathcal{S}$ is convex and closed, property (1) is satisfied and (2) follows from the compact embedding $H^1(\Omega)\hookrightarrow L^2(\Omega)$. 
Hence, we can apply Schauder's fixed point theorem, which assures the existence of a solution $r\in \mathcal{S}$ to  \eqref{equ_lax} with $\tilde{r}$ replaced by $r$.
\end{proof}

\subsection{Entropy dissipation.}
In this subsection we show an entropy dissipation relation for the time discretization \eqref{pme_reg2}. As the entropy density $\tilde{h}$ is convex, we have $\tilde{h}(\varphi_1)-\tilde{h}(\varphi_2)\leq \tilde{h}'(\varphi_1)\cdot(\varphi_1-\varphi_2)$ for all $\varphi_1,\varphi_2\in\mathcal{S}$. If we choose $\varphi_1=r_k$ and $\varphi_2= r_{k-1}$ and using $\tilde{h}'( r_k)= \tilde{u}_k$, we obtain
\begin{align}\label{inequ1}
\frac{1}{\tau}\int_{\Omega}&
 (r_k- r_{k-1})
 \tilde{u}_k\,\ud \bfx\geq\frac{1}{\tau}\int_{\Omega} \left[
\tilde{h}( r_k)-\tilde{h}( r_{k-1})
\right] \ud \bfx.
\end{align}
Inserting \eqref{inequ1} in equation \eqref{pme_reg2} with the test function $\Phi=\tilde{u}_k$ leads to
\begin{equation}
\label{inequ2}
\int_{\Omega}\tilde{h}( r_k)\,\ud \bfx+\tau \! \! \int_{\Omega} \left[ \nabla \tilde{u}_k^T\, n(r_k) \nabla \tilde{u}_k +g(r_k) \nabla \tilde{u}_k\right] \ud \bfx+\tau ^2 R(\tilde{u}_k,\tilde{u}_k)\leq \int_{\Omega}\tilde{h}( r_{k-1})\,\ud \bfx.
\end{equation}

\begin{lemma}\label{lemma2}
Let $ b :\Omega \rightarrow [0,1/\varepsilon_2)$ be such that $\nabla b \in L^2(\Omega)$, $n$ and $g$ as above and let $r \in L^2(\Omega)$, $r \in \mathcal{S}^\circ$ a.e. be such that $\tilde u = \tilde h'(r) \in H^1(\Omega)$. Then there exists a constant $C \geq 0$ depending only on $\varepsilon_1$, $\varepsilon_2$ and $\Vert \nabla b \Vert_{L^2(\Omega)}$ such that \begin{multline} \label{entropyinequality}
\int_\Omega \big[\nabla \tilde u^T n(r) \nabla \tilde u + g(r) \nabla \tilde u 
\big]\,\ud \bfx  \\  \geq \int_{\Omega} \left[\frac{\varepsilon_1}{4}|\nabla  r|^2+\frac{\tau^2}{2} \frac{ \varepsilon_1^3  r^2}{(1-\varepsilon_1 r-\varepsilon_2 b)^2}|\varepsilon_1 \nabla r+\varepsilon_2 \nabla b|^2\right] \ud \bfx - C. 
\end{multline}
\end{lemma}
\begin{proof}
Inserting the definition of $n$ and $g$ we have
\begin{align*}
 \int_{\Omega} \left[ \nabla \tilde{u}^T n(r) \nabla \tilde{u}+ g(r) \nabla \tilde{u}\right] \ud \bfx  =\int_{\Omega}  \left[ r (1 -\varepsilon_2  b) |\nabla \tilde{u}|^2+\frac{\varepsilon_1\varepsilon_2r^2 b}{1+\varepsilon_1 r} |\nabla \tilde{u}|^2+\frac{ \varepsilon_2^2 rb}{1+\varepsilon_1 r}\nabla b \nabla \tilde{u}\right] \ud \bfx.
\end{align*}
Using Young's inequality we deduce that
\begin{align*}
- \frac{ \varepsilon_2^2 rb}{1+\varepsilon_1 r}\nabla b \nabla \tilde{u}\leq \frac{\varepsilon_1\varepsilon_2r^2 b}{1+\varepsilon_1 r} |\nabla \tilde{u}|^2+\frac{\varepsilon_2^3 b}{4\varepsilon_1(1+\varepsilon_1 r)} |\nabla b|^2 \leq  \frac{\varepsilon_1\varepsilon_2r^2 b}{1+\varepsilon_1 r} |\nabla \tilde{u}|^2+\frac{\varepsilon_2^2}{4\varepsilon_1 } |\nabla b|^2.
\end{align*}
Since
\begin{align*}
r (1-\varepsilon_2 b) |\nabla \tilde{u}|^2=r (1 - \varepsilon_1 r-\varepsilon_2 b) |\nabla \tilde{u}|^2+\varepsilon_1 r^2 |\nabla \tilde{u}|^2,
\end{align*}
we obtain
\begin{align*}
\int_{\Omega} \left[ \nabla \tilde{u}^T n(r) \nabla \tilde{u}+ g(r) \nabla \tilde{u}\right] \ud \bfx \geq \int_{\Omega} \varepsilon_1 r^2 |\nabla \tilde{u}|^2 \, \ud \bfx - C_1,
\end{align*}
for some constant $C_1\geq 0$.
Using the definition of $\tilde{u}$ \eqref{u_reg}, Young's inequality to estimate the mixed terms, and
\begin{equation*}
\varepsilon_1 r^2 |\nabla \tilde{u}|^2 =\varepsilon_1 \left|(\varepsilon_1 \nabla r+\varepsilon_2\nabla b)\left( \frac{1}{\varepsilon_1}+r+\tau \frac{r\varepsilon_1}{1-\varepsilon_1 r-\varepsilon_2 b}\right)-\frac{\varepsilon_2}{\varepsilon_1}\nabla b\right|^2,
\end{equation*}
gives
\begin{align*}
\int_{\Omega}  \big[ \nabla \tilde{u}^T n(r) \nabla \tilde{u} &+ g(r) \nabla \tilde{u}\big] \ud \bfx \\ & \geq \int_\Omega \left[ \frac{1}{2\varepsilon_1} |\varepsilon_1 \nabla r+\varepsilon_2 \nabla b|^2+\frac{\tau^2}{2} \frac{ \varepsilon_1^3  r^2}{(1-\varepsilon_1 r-\varepsilon_2 b)^2}|\varepsilon_1 \nabla r+\varepsilon_2 \nabla b|^2\right] \ud \bfx - C_2\\
&\geq \int_\Omega  \left[ \frac{\varepsilon_1}{4}| \nabla r|^2+\frac{\tau^2}{2} \frac{ \varepsilon_1^3  r^2}{(1-\varepsilon_1 r-\varepsilon_2 b)^2}|\varepsilon_1 \nabla r+\varepsilon_2 \nabla b|^2\right] \ud \bfx- C,
\end{align*}
for some constants $C_2,C\geq0$. 
\end{proof}

Applying the dissipation inequality \eqref{entropyinequality} and resolving the recursion \eqref{inequ2} yields
\begin{multline}\label{discrete_entropyinequality}
\int_{\Omega}\tilde{h}( r_k)\,\ud \bfx  +\tau\sum_{j=1}^k\int_{\Omega} \left[ \frac{\varepsilon_1}{4}|\nabla  r_j|^2+\frac{\tau^2}{2} \frac{ \varepsilon_1^3  r_j^2}{(1-\varepsilon_1 r_j-\varepsilon_2 b)^2}|\varepsilon_1 \nabla r_j+\varepsilon_2 \nabla b|^2\right] \ud \bfx\\
+\tau^2\sum_{j=1}^k R (\tilde{u}_j,\tilde{u}_j) \leq \int_{\Omega} \tilde{h}( r_0)\,\ud \bfx+T C.
\end{multline}
This entropy dissipation property is the basis for obtaining sufficient compactness that allows to pass to the limit $\tau \rightarrow 0$ in the next subsection.

\subsection{The limit \texorpdfstring{$\tau \to 0$}{t to 0}.} \label{sec:limit_tau}
In this subsection we perform the limit to a solution of the time-continuous problem. Let $r_k$ be a sequence of solutions to \eqref{pme_reg2}. We define $ r_\tau({\bf x},t)= r_k({\bf x})$ for ${\bf x}\in\Omega$ and $t\in ((k-1)\tau,k\tau]$. Then $r_\tau$ solves the following problem,
\begin{multline}\label{equ1_tau}
\int_0^T  \! \! \int_{\Omega} \left \{ \frac{1}{\tau} (r_\tau-\sigma_\tau  r_\tau)\Phi+
\left [(1+\varepsilon_1 r_\tau-\varepsilon_2 b)\nabla  r_\tau+\varepsilon_2 r_\tau \nabla b \right ]\nabla \Phi \right \} \ud \bfx\,\ud t \\
+\int_0^T  \! \!  \left[\int_{\Omega}\frac{\tau \varepsilon_1}{1-\varepsilon_1 r_\tau -\varepsilon_2 b} n(r_\tau) (\varepsilon_1\nabla r_\tau+\varepsilon_2 \nabla b) \nabla \Phi \,\ud \bfx+\tau R(\Phi,\tilde{u}_\tau)\right] \ud t=0,
\end{multline}
for $\Phi \in L^2(0,T;H^1(\Omega))$. Here $\sigma_\tau$ denotes a shift operator, that is, $(\sigma_\tau  r_\tau)({\bf x},t)= r_\tau ({\bf x},t-\tau)$ for $\tau \leq t\leq T$.
Note that the terms in the second line of \eqref{equ1_tau} are the regularization terms. Using \eqref{equ1_tau}, the inequality \eqref{discrete_entropyinequality} becomes
\begin{multline*}
\int_{\Omega}\tilde{h}( r_\tau(T))\,\ud \bfx + \int_0^T \! \! \int_{\Omega} \left[ \frac{\varepsilon_1}{4}|\nabla  r_\tau|^2+\frac{\tau^2}{2} \frac{ \varepsilon_1^3  r_\tau^2}{(1-\varepsilon_1 r_\tau-\varepsilon_2 b)^2}|\varepsilon_1 \nabla r_\tau+\varepsilon_2 \nabla b|^2\right] \ud \bfx\,\ud t\\
+\tau\int_0^TR (\tilde{u}_\tau,\tilde{u}_\tau)\,\ud t \leq \int_{\Omega} \tilde{h}( r_0)\,\ud \bfx+T C,
\end{multline*}
which leads to the \emph{a priori} estimates in Lemma \ref{lemma3} below. 
\begin{lemma}[\emph{a priori} estimates]\label{lemma3}
There exists a constant $K\in\mathbb{R}^+$ (independent of $\tau$), such that the following bounds hold:
\begin{align}
\| r_\tau\|_{L^2(0,T;H^1(\Omega))}&\leq K,  \label{apriori1}\\
\tau \left\|\frac{ r_\tau}{1-\varepsilon_1 r_\tau-\varepsilon_2 b}(\varepsilon_1 \nabla r_\tau+\varepsilon_2 \nabla b)\right\|_{L^2(\Omega_T)}&\leq K,  \label{apriori2}\\
\sqrt{\tau}\|\tilde{u}_\tau\|_{L^2(0,T;H^1(\Omega))}&\leq K,\label{apriori3}
\end{align}
where $\Omega_T=\Omega \times (0,T)$.
\end{lemma}
\begin{lemma} \label{lemma4}
There exists a constant $K\in\mathbb{R}^+$ (independent of $\tau$), such that the discrete time derivative of $ r_\tau$ is uniformly bounded, that is,
\begin{align}\label{inequ3}
\frac{1}{\tau}\| r_\tau-\sigma_\tau  r_\tau\|_{L^2(0,T;H^1(\Omega)')}&\leq K. 
\end{align}
\end{lemma}

\begin{proof}
Let $\Phi \in L^2(0,T;H^1(\Omega))$. Using Lemma \ref{lemma3} gives  
\begin{align*}
\frac{1}{\tau}\int_0^T \langle  r_\tau-\sigma_\tau  r_\tau,\Phi\rangle\,\ud t =&  -\int_0^T \int_{\Omega}\left[(1+\varepsilon_1 r_\tau-\varepsilon_2 b)\nabla  r_\tau+\varepsilon_2 r_\tau \nabla b\right] \nabla \Phi\,\ud \bfx\,\ud t\\
&-\tau \varepsilon_1\int_0^T \int_{\Omega} \frac{n(r_\tau)}{1-\varepsilon_1 r_\tau-\varepsilon_2 b}(\varepsilon_1 \nabla r_\tau+\varepsilon_2 \nabla b)\nabla \Phi\,\ud \bfx\,\ud t\\
&-\tau \int_0^T  \int_{\Omega}\left( \tilde{u}_\tau\Phi+ \nabla \tilde{u}_\tau\cdot\nabla \Phi\right) \ud \bfx\,\ud t\\
\leq & \, \|(1+\varepsilon_1 r_\tau-\varepsilon_2 b)\|_{L^{\infty}(\Omega_T)}\|\nabla  r_\tau\|_{L^2(\Omega_T)}\|\nabla \Phi\|_{L^2(\Omega_T)}\\
&+\varepsilon_2\| r_\tau \nabla b \|_{L^{\infty}(\Omega_T)}\|\nabla \Phi\|_{L^2(\Omega_T)}\\
&+\tau  \varepsilon_1 \left\|\frac{n(r_\tau)}{1-\varepsilon_1 r_\tau-\varepsilon_2 b} (\varepsilon_1 \nabla r_\tau+\varepsilon_2 \nabla b)\right\|_{L^2(\Omega_T)}\|\nabla \Phi\|_{L^2(\Omega_T)}\\
&+\tau \|\tilde{u}_\tau\|_{L^2(0,T;H^1(\Omega))}\|\Phi\|_{L^2(0,T;H^1(\Omega))}\\
\leq & \,  K\|\Phi\|_{L^2(0,T;H^1(\Omega))},
\end{align*}
as required.
\end{proof}

From Lemmas \ref{lemma3} and \ref{lemma4} we have that $r_\tau\in L^2(0,T;H^1(\Omega))$ and $\frac{1}{\tau}(r_\tau-\sigma_\tau  r_\tau)\in L^2(0,T;H^1(\Omega)')$. Then we can use Aubin--Lions lemma to deduce the existence of a subsequence, also denoted by $r_\tau$, such that, as $\tau\to 0$,
\begin{equation*}
r_\tau \to r \qquad \text{strongly in} \qquad L^2(0,T;L^2(\Omega)).
\end{equation*}
Even though the \emph{a priori} estimates from Lemma \ref{lemma3} are enough to get boundedness for all terms in \eqref{equ1_tau} in $L^2(\Omega_T)$, the compactness results are not enough to identify the correct limits as $\tau \to 0$. In particular, the \emph{a priori} estimates from Lemma \ref{lemma3} and the strong convergence of $r_\tau$ allow us to pass to the correct limit in all the terms except the one resulting from the entropy regularization.
Together with Lemma \ref{lemma3}, we obtain a solution to
\begin{align}
- \int_0^T \! \int_{\Omega}  r \partial_t \Phi \,\ud \bfx\,\ud t = \int_0^T \! \int_{\Omega}J_r \nabla \Phi\,\ud \bfx\,\ud t,
\end{align} 
where 
\begin{align}
&(1+\varepsilon_1 r_\tau-\varepsilon_2 b)\nabla  r_\tau+\varepsilon_2 r_\tau \nabla b+\frac{\tau \varepsilon_1}{1-\varepsilon_1 r_\tau -\varepsilon_2 b} n(r_\tau) (\varepsilon_1\nabla r_\tau+\varepsilon_2 \nabla b)\rightharpoonup J_r \label{limit1},
\end{align}
weakly in  $L^2(\Omega_T)$.
In order to pass to the correct limits in all the terms, we multiply equation \eqref{limit1} by $(1-\varepsilon_1 r_\tau-\varepsilon_2 b)$. Then we obtain that for $\tau\to 0$, we have
\begin{align*}
\frac{\tau \varepsilon_1}{1-\varepsilon_1 r_\tau -\varepsilon_2 b}(1-\varepsilon_1 r_\tau-\varepsilon_2 b)n(r_\tau)(\varepsilon_1\nabla r_\tau+\varepsilon_2 \nabla b)&=\tau\varepsilon_1 n(r_\tau)(\varepsilon_1\nabla r_\tau+\varepsilon_2 \nabla b)\to 0,
\end{align*}
strongly in $L^2(\Omega_T)$.
Since the entropy functional $E_1$ is convex and continuous, it is weakly lower semi-continuous. Because of the weak convergence of $r_\tau(t)$,
\begin{equation*}
\int_\Omega \tilde{h}(r(t))\,\ud \bfx\leq \liminf_{\tau\to 0}\int_\Omega \tilde{h}(r_\tau (t))\,\ud \bfx \qquad \text{for a.e.}\qquad t>0,
\end{equation*}
the limit satisfies the entropy inequality \eqref{theorem1_2}. This completes the proof of Theorem \ref{theorem1}.

\begin{remark}
Note that setting $\varepsilon_2 = 0$ in \eqref{pme}, we obtain an enhanced diffusion equation of the form
\begin{equation}
\partial_t r=\nabla \cdot [(1+\varepsilon_1 r) \nabla r],
\end{equation}
which is a GF with respect to the entropy $E(r)=\int_\Omega  \left( r \log r +\varepsilon_1 r^2/2 \right) \ud \bfx$ and mobility $m(r)=r$. If we pass to the limit $\varepsilon_2 \to 0$ in Theorem \ref{theorem1}, all terms in its proof that depend on $\varepsilon_2$ are well-defined and their limits coincide with setting $\varepsilon_2=0$ in advance, that is, they vanish. 
\end{remark}

\begin{remark}\label{remark_globalexistence}
Theorem \ref{theorem1} also applies to a simpler model, namely the GF induced by the first entropy-mobility pair in \eqref{pme_gf_mod}. In particular, the proof is analogous but simpler: comparing \eqref{pme_gf_mod} with \eqref{pme_gf2} as we do not have the second term in the mobility $n(r)$ and the additional term $g(r)$ is zero. 
\end{remark}

\section{Conclusion.} \label{sec:conclusion}

In this paper we have used the framework of gradient flows (GFs) as well as asymptotic gradient flows (AGFs) to study the existence of a nonlinear Fokker--Planck equation lacking a full gradient flow (GF) structure. The equation describes the diffusion of hard-core interacting particles through a domain with obstacles distributed according to a given porosity function. While this equation can be studied using existing techniques for nonlinear scalar equations, here we used it to showcase how GF techniques can be generalized for AGFs. In a full GF, although there may be different entropies, the equilibrium  is fully determined by its entropy and the dynamic and equilibration behavior can be understood by its interplay with the mobility. In contrast, an AGF is only a GF up to a certain order in a small parameter, and this results in additional freedom when it comes to choosing an entropy-mobility pair.

We performed numerical simulations of the original equation and three choices of AGFs, and compared their steady states and the rates of convergence to these. As an additional way to weight the qualities of each AGF, we also performed stochastic simulations of the underlying  microscopic particle system. By doing so, we could not just compare which AGF structure captures best the behavior of our Fokker--Planck equation, but more importantly, the ``true'' underlying system. We found that, depending on the ratio of self- to obstacles crowding, different AGF structures better captured the system behavior. The simulations also indicated exponential convergence to equilibrium.

We discussed how our equation admits several entropy-mobility pairs to define an AGF, and identified the advantages and disadvantages of each choice when it comes to the analysis. In particular, we found that one pair includes the natural potential term in the entropy coming from the effects of porosity gradients but lacks a bound representing the maximum allowed crowding. As a result, we had to be particularly careful in the existence proof to ensure that solutions do not exceed the maximum packing, which also led to our specific definition of weak solutions. In contrast, an alternative entropy-mobility pair included the correct bounds in the mobility and entropy, leading to better estimates and boundedness by entropy, but lacked the natural potential term in the entropy. 

This work shows that in general it is not possible to determine the entropy and GF structure of a macroscopic equation derived using asymptotic methods from an underlying particle system. We can choose different AGF structures to suit our needs (for example, better entropy dissipations, better bounds), but this will not necessarily ``select'' the AGF that brings physical insight or more accurately captures the behavior of the underlying particle system. This leads to questions such as: How can we propagate the GF structure of the particle system in the derivation, rather than trying to establish it at the level of the macroscopic Fokker--Planck equation? What defines a better AGF? Is there a systematic way to always pick it? Those will be key challenges for future research. 

\section*{Acknowledgments} 

The authors thank Ansgar J\"ungel, Daniel Matthes and Mark Peletier for their constructive comments and remarks. This work was completed during the Workshop  ``Asymptotic gradient flows in multiscale models of interacting particle systems'' that took place in Oxford in July 2017. We thank all participants of the workshop for useful discussions and L'Or\'eal UK and Ireland Fellowship For Women In Science for their financial support. H.R. acknowledges support by the Austrian Science Fund (FWF) project F 65. M.T.W. was partly supported by the Austrian Academy of Sciences \"OAW via the New Frontier's Group NST-001 and by the EPSRC via the First grant EP/P01240X/1.


\end{document}